\documentclass[reqno, 11pt, letterpaper]{amsart}

\usepackage{amsmath}
\usepackage{amsfonts}
\usepackage{amssymb}
\usepackage{graphicx}
\usepackage{amsthm,color,yfonts,cite} \usepackage{paralist}
\usepackage{hyperref}
\usepackage{physics}
\usepackage{dsfont}
\usepackage{soul}

\newtheorem{theorem}{Theorem}
\newtheorem{definition}[theorem]{Definition}
\newtheorem{proposition}[theorem]{Proposition}
\newtheorem{lemma}[theorem]{Lemma}

\theoremstyle{remark}
\newtheorem{remark}[theorem]{Remark}

\newcommand{\R}{\mathbb{R}}

\newcommand{\mcD}{\mathcal{D}}

\usepackage{wrapfig}
\usepackage{tikz}
\usetikzlibrary{arrows,calc,decorations.pathreplacing}
\definecolor{light-gray1}{gray}{0.90}
\definecolor{light-gray2}{gray}{0.80}
\definecolor{light-gray3}{gray}{0.60}

\numberwithin{equation}{section}

\numberwithin{theorem}{section}

\numberwithin{table}{section}

\numberwithin{figure}{section}

\ifx\pdfoutput\undefined
  \DeclareGraphicsExtensions{.pstex, .eps}
\else
  \ifx\pdfoutput\relax
    \DeclareGraphicsExtensions{.pstex, .eps}
  \else
    \ifnum\pdfoutput>0
      \DeclareGraphicsExtensions{.pdf}
    \else
      \DeclareGraphicsExtensions{.pstex, .eps}
    \fi
  \fi
\fi

\title[Non-relativistic limit of scattering states]{The non-relativistic limit of scattering states for the Vlasov equation with short-range interaction potentials}

\date{\today}
\author[Y. Hong]{Younghun Hong}
\address{Department of Mathematics, Chung-Ang University, Seoul 06974, Korea}
\email{yhhong@cau.ac.kr}

\author[S. Pankavich]{Stephen Pankavich}
\address{Department of Applied Mathematics and Statistics, Colorado School of Mines, Golden, CO 80401, USA}
\email{pankavic@mines.edu}

\begin{document}
\begin{abstract}
We study the relativistic and non-relativistic Vlasov equation driven by short-range interaction potentials and identify the large time dynamics of solutions. In particular, we construct global-in-time solutions launched from small initial data and prove that they scatter along the forward free flow to well-behaved limits as $t \to \infty$. Moreover, we prove the existence of wave operators for such a regime and, upon constructing the aforementioned time asymptotic limits, use the wave operator formulation to prove for the first time that the relativistic scattering states converge to their non-relativistic counterparts as $c \to \infty$.
\end{abstract}

\maketitle

\section{Introduction}

We consider the relativistic Vlasov equation
\begin{equation}\label{eq: rVP}
\left.
\begin{aligned}
\partial_t f_c+v_c(p)\cdot\nabla_x f_c+E_c\cdot\nabla_p f_c&=0,\\
f_c(0)&=f^0,
\end{aligned}
\right\}
\end{equation}
for every $1 \leq c \leq \infty$, 
where $f_c=f_c(t,x,p): [0,\infty) \times\mathbb{R}^3\times\mathbb{R}^3\to [0,\infty)$ represents the particle distribution function, $c\geq 1$ is the speed of light, and
$$v_c(p) = \frac{p}{\sqrt{1+\tfrac{|p|^2}{c^2}}}$$ 
is the relativistic velocity function with corresponding inverse, defined for $|q| < c$, given by
$$v_c^{-1}(q) = \frac{q}{\sqrt{1-\frac{|q|^2}{c^2}}}.$$
For $c = \infty$, the relativistic velocity corrections vanish, and we merely define $v_\infty(p) = p$ so that $v_\infty^{-1}(q) = q$. 
Throughout, we have normalized the particle mass for simplicity.
For an integrable function $h: [0,\infty) \times\mathbb{R}^3\times\mathbb{R}^3\to [0,\infty)$, we denote its momentum average by
$$\rho_h(t,x) = \int h(t,x,p) \ dp$$
so that, in particular, the momentum average of the distribution function is
$$\rho_{f_c}(t,x) = \int f_c(t,x,p) \ dp.$$
With this, the corresponding force field is
\begin{equation}\label{eq: rVP, force field}
E_c(t,x)=\beta \nabla w*\rho_{f_c}=\beta \iint_{\mathbb{R}^6}\nabla w(x-y)f_c(t,y,p) dydp
\end{equation}
with $\beta\in\{-1,0,1\}$; $\beta=0$ (free), $\beta=-1$ (repulsive), and $\beta=1$ (attractive). 
Here, $w: \mathbb{R}^3 \to  \mathbb{R}$ represents a given potential function that generates the self-consistent force field.
We assume throughout the paper that for some $\alpha \in (1,2)$ and $C > 0$, the potential satisfies
\begin{equation}
\label{potential_assumption}
|w(x)| \leq C|x|^{-\alpha}, \qquad 
|\nabla w(x)|\leq C|x|^{-(\alpha+1)}
\end{equation}
for $|x|$ sufficiently large. Note that the case $\alpha = 1$ corresponds to the Coulomb potential so that \eqref{eq: rVP} becomes the relativistic Vlasov-Poisson system.
In general, this case of $\alpha \in (1,2)$ is referred to as a \textit{short-range} potential, while $\alpha \in (0,1)$ corresponds to a \textit{long-range} interaction potential.
The former values of $\alpha$ lead to stronger mean field interactions among close particles, which could possibly lead to blow-up of solutions, while the latter values feature weaker short-range interactions, which may lead to slower time-asymptotic decay properties.
In particular, we note that a variety of interaction potentials, including super-Coulombic potentials \cite{HuangKwon}  (i.e., $w(x) \sim |x|^{-\alpha}$) and the well-known Yukawa potential \cite{HaLee}  (i.e., $w(x) \sim |x|^{-1} e^{-a|x|}, \ a > 0$) for screened interactions, satisfy \eqref{potential_assumption}.

As we will study the initial-value problem, we impose the initial condition
$f(0,x,p) = f^0(x,p)$
for $f^0$ given and satisfying a specific smallness condition that we will state later.
We will also use the notation
$$\gamma_c(p) = \sqrt{1 + \frac{|p|^2}{c^2}}$$
to represent the rest momentum, so that $v_c(p) = p/\gamma_c(p)$, and 
denote the derivative of the relativistic velocity by 
\begin{equation}
\label{Adef}
    \mathbb{A}_c(p):=\nabla v_c(p)=\frac{1}{\gamma_c(p)}\mathbb{I}_3-\frac{1}{\gamma_c(p)^3}\bigg(\frac{p_i p_j}{c^2}\bigg)_{i,j=1}^3,
\end{equation}
that is, a $3\times 3$ matrix-valued function for $1 \leq c < \infty$ with $\mathbb{A}_\infty(p) = \mathbb{I}_3$. 
For simplicity, we will further utilize the Japanese bracket notation, namely
$$\langle p \rangle := \sqrt{1 + |p|^2}.$$
Throughout, we will also use the notation
$A(t) \lesssim B(t)$
to represent the fact that there exists a constant $C > 0$, independent of $t \geq 0$, $c\in [1,\infty]$, and small parameters $\eta, \eta_0 > 0$ such that $A(t) \leq C B(t).$
The equation \eqref{eq: rVP} yields the characteristic system of ODEs 
\begin{equation}\label{eq: rVP-ODE}
\left.
\begin{aligned}
\partial_s\big(\mathcal{X}_c(s),\mathcal{P}_c(s)\big)&=\Bigg(v_c\left (\mathcal{P}_c(s)\right ),E_c(s,\mathcal{X}_c(s))\Bigg),\\
(\mathcal{X}_c(t),\mathcal{P}_c(t))&=(x,p),
\end{aligned}
\right \}
\end{equation}
where $(\mathcal{X}_c(s), \mathcal{P}_c(s))=(\mathcal{X}_c(s,t,x,p), \mathcal{P}_c(s,t,x,p))$ is an abbreviated notion for the characteristics that we will employ for the duration of the paper.

In addition to the relativistic system, we consider its non-relativistic analogue, which also satisfies \eqref{eq: rVP} with $c = \infty$ and $v_\infty(p) = p$. Again, $f_\infty=f_\infty(t,x,p): [0,\infty)\times\mathbb{R}^3\times\mathbb{R}^3\to [0,\infty)$ represents the particle distribution function arising from non-relativistic velocities,
$\rho_{f_\infty}(t,x)$ is the momentum average of this quantity, and the force field $E_\infty(t,x)$ is defined as in \eqref {eq: rVP, force field} but with $\rho_{f_c}$ (or $f_c$) replaced by $\rho_{f_\infty}$ (or $f_\infty$).
The non-relativistic system yields the corresponding characteristic ODEs 
\begin{equation}\label{eq: VP-ODE}
\left.
\begin{aligned}
\partial_s\big(\mathcal{X}_\infty(s),\mathcal{P}_\infty (s)\big)&=\Bigg(\mathcal{P}_\infty(s),E_\infty(s,\mathcal{X}_\infty(s))\Bigg),\\
(\mathcal{X}_\infty(t),\mathcal{P}_\infty(t))&=(x,p).
\end{aligned}
\right \}
\end{equation}

Though our study is the first to investigate the convergence of scattering states of the relativistic Vlasov system with short-range potentials to their non-relativistic counterparts, others have rigorously studied properties of this system. Recently, Wang \cite{Wang23} proved global existence and large time decay estimates for small data solutions of the relativistic and non-relativistic Vlasov-Poisson system, which corresponds to the less singular Coulomb potential ($\alpha = 1$).
Additionally, Huang and Kwon established global existence and modified 
scattering of small data solutions of the non-relativistic Vlasov-Riesz system, which includes super-Coulombic potentials, namely $w(x) \sim |x|^{-\alpha}$ for $\alpha \in (1,2)$.
Finally, Ha and Lee proved small data global existence for the relativistic Vlasov-Yukawa system \cite{HaLee}.
That being said, the construction of associated wave operators and the convergence of scattering states in the limit as $c \to \infty$ has not been obtained previously for any of these equations.

\subsection{Outline of Results}
To begin our investigation, we state the main results of the paper. First, we construct global-in-time solutions from sufficiently small initial data.
For brevity, we will use the notation $a^+$ to denote a preselected number which is larger than $a \in \mathbb{R}$ but arbitrarily close to $a$.
\begin{theorem}[Small data solutions]
\label{T1}
Assume
\begin{equation}\label{eq: initial data assumption}
\eta:=\Big\|\langle x\rangle^{3^+}\langle p \rangle^8 f^0\Big\|_{L_{x,p}^\infty} +\Big\|\langle x\rangle^{3^+}\langle p \rangle^9\nabla_{(x,p)}f^0\Big\|_{L_{x,p}^\infty}
\end{equation}
is sufficiently small. Then, for any $1 \leq c \leq \infty$, there exists a unique, global solution $f_c$ satisfying \eqref{eq: rVP} for all $t \in [0,\infty)$ and $(x,p) \in \mathbb{R}^6$. Moreover, the associated force field satisfies the uniform decay bounds
\begin{equation}\label{eq: uniform force field decay bounds}
\sup_{t\geq0}\Big\{(1+t)^{\alpha+1}\|E_c(t)\|_{L_x^\infty(\mathbb{R}^3)}+(1+t)^{\alpha+2}\|\nabla_x E_c(t)\|_{L_x^\infty(\mathbb{R}^3)}\Big\}\lesssim\eta_0,
\end{equation}
for some $\eta_0$ satisfying $0 < \eta < \eta_0 \ll 1$ where the implicit constant in \eqref{eq: uniform force field decay bounds} does not depend on $c\in [1,\infty]$.
\end{theorem}

With solutions in hand for every $1 \leq c \leq \infty$, we study the large time limits of each system for fixed $c$, as in \cite{ChoiHa, ChoiKwon, HuangKwon,IPWW22, Pankavich1, Wang23}.
In particular, we identify limiting characteristics $(\mathcal{X}_c^+(x,p), \mathcal{P}_c^+(x,p))$ and $(\mathcal{X}_\infty^+(x,p), \mathcal{P}_\infty^+(x,p))$ satisfying
$$\Big(\mathcal{X}_c(t,0,x,p)-tv_c\big(\mathcal{P}_c(t,0,x,p)\big), \mathcal{P}_c(t,0,x,p)\Big) \to (\mathcal{X}_c^+(x,p), \mathcal{P}_c^+(x,p))$$
and
$$\Big(\mathcal{X}_\infty(t,0,x,p)-t\mathcal{P}_\infty(t,0,x,p), \mathcal{P}_\infty(t,0,x,p)\Big) \to (\mathcal{X}_\infty^+(x,p), \mathcal{P}_\infty^+(x,p))$$
for all $(x,p) \in \R^6$ as $t \to \infty$.
With this, we further construct limiting distributions
$f_c^+(x,p)$ and $f_\infty^+(x,p)$ in order to establish
$$g_c(t,x,p) = f_c\Big(t,x+tv_c(p),p\Big) \to f_c^+(x,p)$$
and
$$g_\infty(t,x,p) = f_\infty\Big(t,x+tp,p\Big) \to f_\infty^+(x,p)$$
as $t \to \infty$.

\begin{theorem}[Scattering]
\label{T2}
For any $1 \leq c \leq \infty$, let $f_c$ be the unique, global solution constructed within Theorem \ref{T1}.
Then, for any $1 \leq c \leq \infty$, there exists $(\mathcal{X}_c^+,  \mathcal{P}_c^+) \in C(\mathbb{R}^6)$ and non-negative $f_c^+ \in L^1(\mathbb{R}^6)$ such that 
$$\Big(\mathcal{X}_c(t,0,x,p)-tv_c\big(\mathcal{P}_c(t,0,x,p)\big), \mathcal{P}_c(t,0,x,p)\Big) \to \big(\mathcal{X}_c^+(x,p), \mathcal{P}_c^+(x,p)\big)$$
and
$$f_c\Big(t,x+tv_c(p),p\Big) \to f_c^+(x,p)$$
as $t \to \infty$ for all $(x,p) \in \R^6$  with the convergence estimate
$$ \Big\Vert f_c\Big(t,x+tv_c(p),p\Big) - f_c^+(x,p)\Big\Vert_{L^1_{x,p}(\mathbb{R}^6)} \lesssim \frac{1}{(1+t)^\alpha}\|\nabla_{(x,p)}f^0\|_{L_{x,p}^1(\mathbb{R}^6)}$$
for all $t \geq 0$.
\end{theorem}

Finally, our main result entails the convergence of relativistic scattering states to their non-relativistic counterparts as $c \to \infty$.
\begin{theorem}[Non-relativistic limit]
\label{T3}
Let $f_c^+, f_\infty^+ \in L^1\left (\mathbb{R}^6 \right)$ be the time asymptotic limits constructed within Theorem \ref{T2}. Then, the induced scattering states of the relativistic system converge to those of the non-relativistic system as $c \to \infty$. More specifically, for all $1 \leq c \leq \infty$, we have the convergence estimate
$$\|f_c^+-f_\infty^+\|_{L_{x,p}^1(\mathbb{R}^6)}\lesssim\frac{1}{c^2}\|\langle p \rangle^3\nabla_{(x,p)}f^0\|_{L_{x,p}^1(\mathbb{R}^6)}.$$
Similarly, the respective fields and characteristic flows converge as $c \to \infty$ with the same order $\mathcal{O}\left (c^{-2} \right )$ of convergence (see Proposition \ref{prop: nr limit for Vlasov}).
\end{theorem}

Hence, not only do solutions of the relativistic system converge to their non-relativistic counterparts on finite time intervals as $c \to \infty$ (as in \cite{Degond, Schaeffer} for the relativistic Vlasov-Maxwell system), but the limiting states (as $t\to \infty$) of the relativistic system further converge to those of the non-relativistic system in the classical limit as $c \to \infty$.

\begin{remark}
We have not attempted to significantly reduce moments in our norms but instead have focused on obtaining what is believed to be the correct rate of convergence of the scattering states, namely $\mathcal{O}(c^{-2})$.
\end{remark}
\begin{remark}
Using new tools developed in \cite{Bigorgne, Breton, PBA} to study the large time asymptotic behavior of solutions to the relativistic Vlasov-Maxwell system, it may be possible to extend many aspects of our proofs to show that the recently-discovered scattering states of small data solutions of that system converge as $c \to \infty$ to limiting states of the non-relativistic Vlasov-Poisson system at the same order, but this currently remains an open problem.
\end{remark}

\subsection{Quantum mechanical analogy of the main results}

By the quantum-classical correspondence, the relativistic Vlasov equation \eqref{eq: rVP} corresponds to the semi-relativistic Hartree equation
\begin{equation}\label{eq: semi-relativisitic Hartree equation}
i\hbar\partial_t \gamma^\hbar(t)=\Big[\Big(\sqrt{c^4-c^2\hbar^2\Delta}-c^2\Big)+\beta w*\rho^\hbar_{\gamma^\hbar}(t), \gamma^\hbar(t)\Big]
\end{equation}
describing the dynamics of relativistic quantum particles in the Heisenberg picture. In \eqref{eq: semi-relativisitic Hartree equation}, $\hbar>0$ represents the reduced Planck constant, the unknown $\gamma^\hbar(t): I(\subset\mathbb{R})\to L^2(\mathbb{R}^3)$ is an operator-valued quantum observable, $\rho_\gamma^\hbar=(2\pi\hbar)^3K_{\gamma}(x,x)$ is the total density where $K_\gamma(x,y)$ is the integral kernel of $\gamma$, and $[A,B]=AB-BA$ is the Lie bracket. For fixed $c>0$, the classical equation \eqref{eq: rVP} has been rigorously derived from the quantum one \eqref{eq: semi-relativisitic Hartree equation} via the semi-classical limit $\hbar\to0$, including the case of Coulomb interactions \cite{AkiMarkowichSparber08, DietlerRademacherSchlein18, LeoplodSaffirio23}.

By this correspondence, one may also expect that the quantum and classical models would share similar dynamical properties. Indeed, for the quantum model \eqref{eq: semi-relativisitic Hartree equation}, the long-time dynamics of small data states has been studied focusing on the single particle case with normalized coefficients, namely
\begin{equation}\label{eq: semi-relativisitic Hartree equation-one particle}
i\partial_t\phi=\big(\sqrt{1-\Delta}-1\big)\phi+\beta \big(w*|\phi|^2\big) \phi=0,
\end{equation}
where $\phi=\phi(t,x): I(\subset\mathbb{R})\times\mathbb{R}^3\to\mathbb{C}$. When $w$ is a short-range potential, global well-posedness and scattering of small-data solutions have been obtained \cite{ChoOzawa06, ChoOzawaSasakiShim09, HayahiNaumkinOgawa15, HerrTesfahun15, Yang19}. Our first two main theorems (Theorem \ref{T1} and Theorem \ref{T2}) are their classical analogues with a decay bound that holds uniformly for $1 \leq c\leq \infty$. On the other hand, in the case of long-range interaction potentials, a modification of the limiting profile and wave operators is required, which involves the limiting potential or force field \cite{Pusateri14}. Its classical analogue will be considered in future work.

In addition, the non-relativistic limit has been studied from the nonlinear Klein-Gordon equation to the nonlinear Schr\"odinger equation \cite{MachiharaNakanishiOzawa02, MasmoudiNakanishi02}. In particular, the non-relativistic limits of the wave operator and the scattering operator have been established \cite{Nakanishi}. Our last main theorem (Theorem \ref{T3}) is related to this result in some sense.

\subsection{Outline of the proofs}

Recently, the asymptotic behavior for kinetic equations has been addressed by adjusting the approaches and tools developed for nonlinear dispersive equations \cite{ChoiHa, ChoiKwon, HJT, HuangKwon, IPWW22, Wang23}. In this article, we follow this point of view in a broad sense, but unlike many of the aforementioned works, our analysis is strongly based on a Lagrangian approach via the method of characteristics, that is, a well-known method introduced by Bardos and Degond \cite{BD}.

Within the proof, one of the key new ingredients, motivated by quantum theory, is the use of the \textit{classical finite-time wave operator}
\begin{equation}\label{eq: Wc(t) definition}
\mathcal{W}_c(t):=\Phi_c^{\textup{free}}(t)^{-1}\circ\Phi_c(t):\mathbb{R}^6\to\mathbb{R}^6,
\end{equation}
where $\Phi_c(t)$ denotes the relativistic Hamiltonian flow associated with the vector field $E_c(t)$ and $\Phi_c^{\textup{free}}(t)$ is the relativistic free flow, and the limiting wave operator 
$$\mathcal{W}_c^+:=\lim_{t\to\infty} \mathcal{W}_c(t):\mathbb{R}^6\to\mathbb{R}^6$$
(see Sections \ref{sec: Finite-time classical wave operator} and \ref{sec: Limiting classical wave operator} for definitions and basic properties). It is a classical analogue of the quantum wave operator, that is, a well-known tool in quantum linear scattering theory, given by 
$$W_V^+:=\lim_{t\to+\infty}e^{-it\Delta} e^{it(\Delta-V)}P_c: L^2(\mathbb{R}^3)\to P_cL^2(\mathbb{R}^3),$$
where $P_c$ is the spectral projection of $-\Delta+V$ on the continuous spectrum. In fact, for classical dynamical models, the composition of the backward free flow and the forward perturbed flow \eqref{eq: Wc(t) definition} has been used in the study of long-time asymptotics but in the form of the characteristic equation \cite{Pankavich2, Pankavich1, Pankavich3}. Nevertheless, the wave operator formulation turns out to have several crucial advantages in our setting, as described below.

For the main theorems, a key first step is to establish the dispersion estimates for perturbed linear flows (Propositions \ref{prop: density function estimates for perturbed relativistic flows} and \ref{prop: derivative bounds for perturbed relativistic flows}), which are employed in Section \ref{sec: {Uniform bounds and scattering for the relativistic Vlasov equation}} to construct the solution to the relativistic Vlasov equation \eqref{eq: rVP} with uniform decay bounds \eqref{eq: uniform force field decay bounds} for the force field. For the proof, we use the wave operator to express $f_c(t,x,p)$ as $f^0(\mathcal{W}_c(t)^{-1}(x-tv_c(p),p))$. Then, one can obtain the desired bounds combining the dispersion estimate from the backward-in-time free flow $(x-tv_c(p),p)$ and boundedness of the wave operator. We note that this approach is natural in its quantum mechanical analogue. Indeed, in \cite{Yajima}, Yajima established the boundedness of the wave operator $W_V^\pm$ in the Sobolev space $W^{k,p}(\mathbb{R}^d)$ for any $k\geq 0$ and $1\leq p\leq\infty$. Hence, by the intertwining property
$$ e^{it(\Delta-V)}P_c=(W_V^+)^*e^{it\Delta}W_V^+,$$ the $L^1 \to L^\infty$-bound for the perturbed flow $e^{i(\Delta-V)}P_c$ is obtained from that for the free flow $e^{it\Delta}$ (see \cite[Theorem 1.3]{Yajima}). In a similar context, uniform decay estimates for the nonlinear Hartree equation in the semi-classical regime are also obtained by proving uniform bounds for the associated wave operator \cite{Hadama, HadamaHong}.

Once the nonlinear solutions are constructed with uniform bounds \eqref{eq: uniform force field decay bounds} in Section \ref{GE}, we obtain suitable uniform bounds for the associated wave operators (Lemma \ref{lemma:  the wave operator, almost identity} and \ref{lemma: explicit formula for the classical wave operator}), but we also prove scattering along the forward free flow (Theorem \ref{T2}) in Section \ref{sec: scattering}. 

For the last main result (Theorem \ref{T3}), one can see that it is quite complicated to show the non-relativistic limit of the scattering states $f_c^+(x,p)\to f_\infty^+(x,p)$ if one compares them directly at the PDE level, because the limits of $f_c(x+tv_c(p),p)$ and $f_c(x+tp,p)$ are determined implicitly. A simple but important observation is that in using wave operators the scattering states possess clear representations 
$$f_c^+(x,p)=f^0\big((\mathcal{W}_c^+)^{-1}(x,p)\big) \qquad \mathrm{and} \qquad f_\infty^+(x,p)=f^0\big((\mathcal{W}_\infty^+)^{-1}(x,p)\big),$$
and thus, one can prove convergence using the non-relativistic limit of the wave operator $\mathcal{W}_c^+(x,p)\to \mathcal{W}_\infty^+(x,p)$ (Proposition \ref{prop: nr limit for Vlasov}) and its boundedness properties. We also note that the wave operator can be expressed as the limit of the characteristic flow (see Lemma \ref{lemma: explicit formula for the classical wave operator}). Hence, in this way, a more complicated PDE analysis can be reduced to a much simpler and more explicit ODE analysis. Based on these observations, we establish the convergence of the scattering states in Section \ref{sec: Non-relativistic limit for the Vlasov equation and scattering states}.

\subsection{Organization of the paper}

In the next section, we will briefly establish some preliminary lemmas concerning the relativistic velocity function, momentum averages along the backward free flow, and the behavior of the characteristic flow that will be used throughout the paper.
Section $3$ is then dedicated to formulating wave operators for the classical system that mirror the current dynamical understanding for quantum systems, e.g. the Hartree equation, and identifying their properties.
Additionally, crucial decay estimates of the density are contained within this section.
The subsequent section focuses on the existence of solutions launched by small initial data and ultimately shows that they obey uniform (in $c$ and $t$) decay estimates on the force field, namely \eqref{eq: small force field assumption}.
With solutions in hand for every $1 \leq c \leq \infty$, we then establish, within the same section, the large time limits of each system.
Finally, in Section $5$, we obtain the non-relativistic limit (similar to \cite{Degond, Schaeffer} for the Vlasov-Maxwell system) for the characteristics on large time intervals, namely 
$$(\mathcal{X}_c(t), \mathcal{P}_c(t)) \to (\mathcal{X}_\infty(t), \mathcal{P}_\infty(t))$$ 
for $t \geq 1$ as $c \to \infty$.
The (uniform) convergence of the characteristics implies
$$g_c(t,x,p) \to g_\infty(t,x,p)$$
in $L^1_{x,p}(\mathbb{R}^6)$ as $c \to \infty$
and further yields convergence of the field $E_c(t,x) \to E_\infty(t,x)$.
The section concludes with the proof of our main result. More specifically, the previously constructed limits are used to show
that the scattering states of the relativistic system converge to those of the non-relativistic system as $c \to \infty$.
In this direction, we prove
$$(\mathcal{X}_c^+(x,p), \mathcal{P}_c^+(x,p)) \to (\mathcal{X}_\infty^+(x,p), \mathcal{P}_\infty^+(x,p))$$
and similarly
$$f_c^+(x,p) \to f_\infty^+(x,p)$$
as $c \to \infty$.

\subsection{Acknowledgment}
Y. Hong was supported by a National Research Foundation of Korea (NRF) grant funded by the Korean government (MSIT) (No. RS-2023-00208824, No. RS-2023-00219980).
S. Pankavich was supported by the US National Science Foundation under award DMS-2107938.
 
\section{Preliminaries}

\subsection{Properties of the relativistic velocity}

We first recall \eqref{Adef} and note that the partial derivatives of the $(i,j)$-component of $\mathbb{A}_c(p)$ are given by 
\begin{equation}
\label{Derivatives of A}
    \partial_{p_k}\mathbb{A}_c^{ij}(p)=-\frac{1}{c}\frac{1}{\gamma_c(p)^3}\bigg(\delta_{ij}\frac{p_k}{c}+\delta_{ik}\frac{p_j}{c}+\delta_{jk}\frac{p_i}{c}-\frac{3}{\gamma_c(p)^2}\frac{p_i p_j p_k}{c^3}\bigg)
    \end{equation}
for any $k=1,2,3$ and $1 \leq c < \infty$.
Furthermore, $\partial_{x_k} \mathbb{A}^{ij}_\infty(p) = 0$ for any $k=1,2,3$.

\begin{lemma}[Bounds on the relativistic velocity and its derivatives]\label{lemma: bounds for the relativistic velocity and its derivatives}
For every $p \in \mathbb{R}^3$, we have
\begin{align}
|v_c(p)-p|&\leq|p|\min\bigg\{1,\frac{|p|^2}{c^2}\bigg\},\label{eq: vc asymptotic}\\
\bigg\|\mathbb{A}_c(p)-\frac{1}{\gamma_c(p)}\mathbb{I}_3\bigg\|&\leq\frac{1}{\gamma_c(p)}\min\bigg\{1,\frac{|p|^2}{c^2}\bigg\},\label{eq: Ac asymptotic}\\
\bigg\|\mathbb{A}_c(p)-\mathbb{I}_3\bigg\|&\lesssim\min\bigg\{1,\frac{|p|^2}{c^2}\bigg\},\label{eq: Ac asymptotic'}\\
|\partial_{p_k}\mathbb{A}_c^{ij}(p)|&\lesssim\frac{1}{c\gamma_c(p)^2}\min\bigg\{1,\frac{|p|}{c}\bigg\},\label{eq: Ac  derivative asymptotic}
\end{align}
where $\|\cdot\|$ is the matrix norm.
\end{lemma}

\begin{proof}
First, note that 
$$0\leq1-\frac{1}{\gamma_c(p)}=\frac{\gamma_c(p)^2-1}{\gamma_c(p)(\gamma_c(p)+1)}=\frac{|p|^2/c^2}{\gamma_c(p)(\gamma_c(p)+1)} \leq \frac{|p|^2}{c^2} \gamma_c(p)^{-2},
$$
and as $$\gamma_c(p)^2 \geq \max \left \{1, \frac{|p|^2}{c^2} \right \},$$
we find
$$\left |\frac{1}{\gamma_c(p)}-1 \right |\leq \min\left \{1,\frac{|p|^2}{c^2} \right \}.$$
Hence, the bound for $|v_c(p)-p|$ follows. The other three bounds follow from this estimate, the formula for $\mathbb{A}_c(p)$, and the derivative formula \eqref{Derivatives of A}. 
\end{proof}

For the proof of the non-relativistic limit, the interpolated velocity, defined by
\begin{equation}\label{eq: interpolated velocity}
v_c^\theta(p):=\theta v_c(p)+(1-\theta)p,
\end{equation}
for $0\leq\theta\leq 1$,
arises naturally. Note that $v_c^0(p)=p$ is the non-relativistic velocity while $v_c^1(p)=v_c(\theta)$ is the relativistic velocity. Moreover, the derivative of $v_c^\theta(p)$ is a $3\times3$ symmetric matrix, precisely,
$$\nabla v_c^\theta(p)=\theta\mathbb{A}_c(p)+(1-\theta)\mathbb{I}_3=\bigg(\frac{\theta}{\gamma_c(p)}+1-\theta\bigg)\mathbb{I}_3-\frac{\theta}{\gamma_c(p)^3}\bigg(\frac{p_i p_j}{c^2}\bigg)_{i,j=1}^3.$$
Its spectral properties are given by the following lemma.

\begin{lemma}\label{lemma: Ac eigenvalues}
The matrix $\nabla v_c^\theta(p)$ has 3 eigenvalues: $\frac{\theta}{\gamma_c(p)}+(1-\theta)$ of multiplicity $2$ and $\frac{\theta}{\gamma_c(p)^3}+(1-\theta)$. Thus, $\nabla v_c^\theta(p)$ is symmetric positive-definite with
$$\textup{det}(\nabla v_c^\theta(p))=\bigg(\frac{\theta}{\gamma_c(p)}+(1-\theta)\bigg)^2\bigg(\frac{\theta}{\gamma_c(p)^3}+(1-\theta)\bigg).$$ 
\end{lemma}

\begin{proof}
It is obvious that the lemma holds when $p=0$. Suppose that $p\neq0$. By symmetry, we may assume that $p_1\neq 0$. For convenience, let $z=\frac{\theta}{\gamma_c(p)}+1-\theta$. Then, by elementary calculations, we obtain
$$\begin{aligned}
\textup{det}\Big(\nabla v_c^\theta(p)-\lambda\mathbb{I}_3\Big)&=\textup{det}\begin{bmatrix}z-\frac{\theta p_1^2}{c^2\gamma_c(p)^3}-\lambda& -\frac{\theta p_1 p_2}{c^2\gamma_c(p)^3}& -\frac{\theta p_1 p_3}{c^2\gamma_c(p)^3} \\ -\frac{\theta p_1 p_2}{c^2\gamma_c(p)^3}& z-\frac{\theta p_2^2}{c^2\gamma_c(p)^3}-\lambda& -\frac{\theta p_2 p_3}{c^2\gamma_c(p)^3} \\ -\frac{\theta p_1 p_3}{c^2\gamma_c(p)^3}& -\frac{\theta p_2 p_3}{c^2\gamma_c(p)^3}& z-\frac{\theta p_3^2}{c^2\gamma_c(p)^3}-\lambda \end{bmatrix}\\
&=\textup{det}\begin{bmatrix}z-\frac{\theta p_1^2}{c^2\gamma_c(p)^3}-\lambda& -\frac{\theta p_1 p_2}{c^2\gamma_c(p)^3}& -\frac{\theta p_1 p_3}{c^2\gamma_c(p)^3}  \\ -\frac{p_2}{p_1}(z-\lambda)& z-\lambda& 0 \\ -\frac{p_3}{p_1}(z-\lambda)& 0& z-\lambda \end{bmatrix}\\
&=(z-\lambda)^2\textup{det}\begin{bmatrix}z-\frac{\theta p_1^2}{c^2\gamma_c(p)^3}-\lambda& -\frac{\theta p_1 p_2}{c^2\gamma_c(p)^3}& -\frac{\theta p_1 p_3}{c^2\gamma_c(p)^3} \\ -\frac{p_2}{p_1}& 1& 0 \\ -\frac{p_3}{p_1}& 0& 1 \end{bmatrix}\\
&=(z-\lambda)^2\textup{det}\begin{bmatrix}z-\frac{\theta |p|^2}{c^2\gamma_c(p)^3}-\lambda& 0& 0 \\ -\frac{p_2}{p_1}& 1& 0 \\ -\frac{p_3}{p_1}& 0& 1 \end{bmatrix}=(z-\lambda)^2\bigg(z-\frac{\theta |p|^2}{c^2\gamma_c(p)^3}-\lambda\bigg).
\end{aligned}$$
Thus, $z-\frac{\theta |p|^2}{c^2\gamma_c(p)^3}=\frac{\theta}{\gamma_c(p)^3}+(1-\theta)$ and $z=\frac{\theta}{\gamma_c(p)}+1-\theta$ (of multiplicity 2) are eigenvalues of $\nabla v_c^\theta(p)$, and the desired result follows.
\end{proof}

\subsection{Basic inequalities}

Finally, we state the basic inequalities that will be needed for the proof of Theorem \ref{T1}.  
First, we prove a dispersive estimate for the free flow in the following form. For $0\leq\theta\leq 1$ and $h=h(t,x,p):[0,\infty)\times\mathbb{R}^3\times\mathbb{R}^3\to[0,\infty)$, we define 
\begin{equation}\label{eq: T[g]}
T^\theta[h](t,x) 
:= \int_{\mathbb{R}^3} h \Big(t, x-tv_c^\theta(p),p\Big) dp,
\end{equation}
where $v_c^\theta(p)$ is given by \eqref{eq: interpolated velocity}.

\begin{lemma}[Dispersive bounds for the free flow associated with $v_c^\theta(p)$]\label{lemma: dispersive bounds for the free flow associated with interpolated v}
For all $t\geq0$, we have
\begin{equation}\label{eq: dispersive bounds for the free flow associated with interpolated v-1}
\big\|T^\theta[h](t,x)\big\|_{L_x^1(\mathbb{R}^3)}\leq \|h(t,x,p)\|_{L_{x,p}^1(\mathbb{R}^6)}
\end{equation}
and
\begin{equation}
\label{eq: dispersive bounds for the free flow associated with interpolated v-2}
\big\|T^\theta[h](t,x)\big\|_{L_x^\infty(\mathbb{R}^3)}\lesssim \frac{1}{(1+t)^3}\Big\|\langle x\rangle^{3^+}\langle p\rangle^5 h(t,x,p)\Big\|_{L_{x,p}^\infty(\mathbb{R}^6)}.
\end{equation}
\end{lemma}

\begin{proof}
The inequality \eqref{eq: dispersive bounds for the free flow associated with interpolated v-1} is trivial. 
Moreover, it is clear that 
$$
\big|T^\theta[h](t,x)\big|\lesssim 
\Big\|\langle p\rangle^5 h \Big(t, x-tv_c^\theta(p),p\Big)\Big\|_{L_{x,p}^\infty}=\big\|\langle p\rangle^5 h(t, x,p)\big\|_{L_{x,p}^\infty}.$$
Hence, it suffices to show that 
\begin{equation}\label{eq: dispersive bound for the free flow}
\big\|T^\theta[h](t,x)\big\|_{L_x^\infty}\lesssim \frac{1}{t^3}\big\|\langle x\rangle^{3^+}\langle p\rangle^5h(t,x,p)\big\|_{L_{x,p}^\infty}.
\end{equation}
Indeed, for fixed $(t,x)\in [0,\infty)\times\mathbb{R}^3$, the function $p\mapsto y=x-tv_c^\theta(p):\mathbb{R}^3\to \mathcal{R}_{t,x}$ is one-to-one, where $\mathcal{R}_{t,x}\subset \mathbb{R}^3$ denotes the image of the map $p\mapsto y$. Thus, we denote the inverse of $p\mapsto y$ by $p=p(y)=(v_c^\theta)^{-1}(\frac{x-y}{t})$. Then, using Lemma \ref{lemma: Ac eigenvalues} to implement the change of variables $p=p(y): \mathcal{R}_{t,x}\to\mathbb{R}^3$ with the associated Jacobian 
$$\begin{aligned}
\big|\textup{det}\nabla_yp(y)\big|&=\big|\textup{det}\nabla_p y\big|^{-1}= \Big|\textup{det}\Big(t\nabla v_c^\theta(p)\Big)\Big|^{-1}=\frac{1}{t^3(\frac{\theta}{\gamma_c(p)}+(1-\theta))^2(\frac{\theta}{\gamma_c(p)^3}+(1-\theta))}\\
&\leq \frac{\gamma_c(p)^5}{t^3}=\frac{\gamma_c(p(y))^5}{t^3},
\end{aligned}$$
we obtain
$$\begin{aligned}
\big|T^\theta[h](t,x)\big| 
&=\bigg|\int_{\mathcal{R}_{t,x}} h \big(t,y, p(y)\big) \big|\textup{det}\nabla_y p(y)\big| dy\bigg|\leq\frac{1}{t^3}\int_{\mathbb{R}^3} \Big(\gamma_c(p)^5 h\Big)\big(t,y, p(y) \big) dy\\
&\lesssim\frac{1}{t^3}\Big\|\langle x\rangle^{3^+}\gamma_c(p)^5 h(t,x,p)\Big\|_{L_{x,p}^\infty},
\end{aligned}$$
which gives \eqref{eq: dispersive bound for the free flow}
upon noting $\gamma_c(p) \leq \langle p \rangle$ as $c \geq 1$.
\end{proof}

We also recall the well-known interpolation inequality.

\begin{lemma}[Interpolation inequality]
\label{lemma: interpolation inequality}
Assume $h \in L^1(\mathbb{R}^3) \cap L^\infty(\mathbb{R}^3)$. If $|\nabla w(x)|\lesssim\frac{1}{|x|^{\alpha+1}}$ for $0<\alpha<2$, then 
$$\|\nabla w*h\|_{L^\infty(\mathbb{R}^3)}\lesssim \|h\|_{L^1(\mathbb{R}^3)}^{\frac{2-\alpha}{3}}\|h\|_{L^\infty(\mathbb{R}^3)}^{\frac{\alpha+1}{3}}.$$
\end{lemma}

\begin{proof}
For any $R>0$, we have
$$\begin{aligned}
|\nabla w*h(x)|&\lesssim\int_{|x-y|\leq R} \frac{h(y)}{|x-y|^{\alpha+1}}dy +\int_{|x-y|>R} \frac{h(y)}{|x-y|^{\alpha+1}}dy\\
&\lesssim R^{2-\alpha}\|h\|_{L^\infty}+R^{-(\alpha+1)}\|h\|_{L^1}.
\end{aligned}$$
Hence, taking $R = \left (\Vert h \Vert_1 / \Vert h \Vert_\infty \right )^\frac{1}{3}$ to optimize the bound, the proof is complete.
\end{proof}

\subsection{Analysis of the characteristic flow}

Suppose that the vector field $E=E(t,x): [0,\infty) \times \mathbb{R}^3\to\mathbb{R}^3$ is small and decays fast enough in time; precisely, there exist $\alpha>1$ and sufficiently small $0<\eta_0 \ll 1$ such that 
\begin{equation}\label{eq: small force field assumption}
\sup_{t\geq0}\Big((1+t)^{\alpha+1}\|E(t)\|_{L_x^\infty(\mathbb{R}^3)}+(1+t)^{\alpha+2}\|\nabla _xE(t)\|_{L_x^\infty(\mathbb{R}^3)}\Big)\leq\eta_0.
\end{equation}
For $1\leq c\leq\infty$, including the non-relativistic case $c=\infty$, we are concerned with the characteristic flow
$$\Xi_c(s,t,x,p)=\big(\mathcal{X}_c(s,t,x,p), \mathcal{P}_c(s,t,x,p)\big)$$
solving the Hamiltonian ODE
\begin{equation}\label{eq: r-ODE}
\left.
\begin{aligned}
\partial_s\Xi_c(s,t,x,p)&=\big(v_c(\mathcal{P}_c(s,t,x,p)), E(s,\mathcal{X}_c(s,t,x,p))\big),\\
\Xi_c(t,t,x,p)&=(x,p)\in\mathbb{R}^3\times\mathbb{R}^3,
\end{aligned}
\right \}
\end{equation}
which can be written in the integral form as
\begin{equation}\label{eq: integral r-ODE}
\left. \begin{aligned}
\mathcal{X}_c(s,t,x,p)&=x-\int_s^t v_c\big(\mathcal{P}_c(\tau,t,x,p)\big) d\tau,\\
\mathcal{P}_c(s,t,x,p)&=p-\int_s^tE\big(\tau,\mathcal{X}_c(\tau,t,x,p)\big)d\tau.
\end{aligned}\right \}
\end{equation}
For convenience, we only consider positive times $s,t\geq0$. By the smallness assumption \eqref{eq: small force field assumption}, the flow map $\Xi_c(s,t,x,p)$ can be considered as a perturbation of the free flow 
$$\Xi_c^{\textup{free}}(s,t,x,p)=\big(x-(t-s)v_c(p),p\big)$$
when $E(t,x)\equiv 0$. 

\begin{remark}
In Section \ref{GE} we will show that, due to the small data assumption \eqref{eq: initial data assumption}, the force field $E_c(t,x)$ for the Vlasov equation given within Theorem \ref{T1} satisfies the decay condition \eqref{eq: small force field assumption}.
\end{remark}

The following lemma shows that the momentum remains nearly invariant under the flow. This simple lemma will be frequently used throughout the paper.
\begin{lemma}[Perturbed momentum]
\label{lemma: perturbed momentum}
Assume that the force field $E(t,x)$ satisfies the decay bound \eqref{eq: small force field assumption} with $\alpha > 1$ and $0< \eta_0 \ll 1$ sufficiently small. If the backward characteristic flow $\Xi_c$ solves \eqref{eq: r-ODE}, then for all $0\leq s\leq t$,
we have
\begin{equation}\label{eq: bound for Pc}
\left | \mathcal{P}_c(s,t,x,p)- p \right | \lesssim \eta_0,
\end{equation}
where the implicit constant is independent of $c \geq 1$ and $s,t \geq 0$.
Moreover, if $|p-p'|\lesssim\eta_0$, then
\begin{align}
\left |\gamma_c(p') - \gamma_c(p) \right | & \lesssim \frac{\eta_0}{c},\label{eq: gamma c difference bound}\\
\left \Vert \mathbb{A}_c(p') - \mathbb{A}_c(p) \right \Vert & \lesssim \frac{\eta_0}{c}\gamma_c(p)^{-2}.\label{eq: Ac difference bound}
\end{align}
Finally, combining equations \eqref{eq: Ac asymptotic} and \eqref{eq: Ac difference bound} yields
\begin{equation}\label{A bound}
\Vert \mathbb{A}_c(\mathcal{P}_c(s)) \Vert \lesssim \gamma_c(p)^{-1}
\end{equation}
for any $ s \geq 0.$
\end{lemma}

\begin{proof}
Equation \eqref{eq: bound for Pc} follows immediately from \eqref{eq: small force field assumption} and the second equation in \eqref{eq: integral r-ODE}. 
Suppose that $|p-p'|\lesssim\eta_0$. Then, by elementary calculations, we find
$$|\gamma_c(p)-\gamma_c(p')|=\frac{|\gamma_c(p)^2-\gamma_c(p')^2|}{\gamma_c(p)+\gamma_c(p')}\leq\frac{\frac{|p-p'|}{c}(\frac{|p|}{c}+\frac{|p'|}{c})}{\gamma_c(p)+\gamma_c(p')}\lesssim\frac{\eta_0}{c}.$$
Hence, using \eqref{eq: Ac asymptotic} it follows that 
$$\begin{aligned}
\|\mathbb{A}_c(p)-\mathbb{A}_c(p')\|&\leq\frac{1}{\gamma_c(p)}\|\gamma_c(p)\mathbb{A}_c(p)-\gamma_c(p')\mathbb{A}_c(p')\|+\frac{1}{\gamma_c(p)}\|(\gamma_c(p')-\gamma_c(p))\mathbb{A}_c(p')\|\\
&\lesssim\frac{1}{\gamma_c(p)}\bigg\|\frac{1}{\gamma_c(p)^2}\bigg(\frac{p_j p_k}{c^2}\bigg)_{j,k=1}^3-\frac{1}{\gamma_c(p')^2}\bigg(\frac{p_j' p_k'}{c^2}\bigg)_{j,k=1}^3\bigg\|+\frac{\eta_0}{c\gamma_c(p)}\|\mathbb{A}_c(p')\|\\
&\lesssim\frac{\eta_0}{c\gamma_c(p)^2},
\end{aligned}$$
which proves \eqref{eq: Ac difference bound}.
\end{proof}

\section{Wave operator formulation}\label{sec: Wave operator formulation}

Throughout this section, we assume the smallness of the field, namely \eqref{eq: small force field assumption}. 
Under this assumption, we introduce the classical wave operator corresponding to the quantum wave operator in the linear scattering theory.

\subsection{Finite-time classical wave operator}\label{sec: Finite-time classical wave operator}

Suppose that \eqref{eq: small force field assumption} holds, and we define the one-parameter group 
$$\Phi_c(t):=\Xi_c(t,0,x,p)=\big(\mathcal{X}_c(t,0, x,p), \mathcal{P}_c(t,0, x,p)\big): \mathbb{R}^3\times \mathbb{R}^3\to\mathbb{R}^3\times \mathbb{R}^3$$
as the initial data-to-solution map for the characteristic ODE
\begin{equation}\label{eq: forward characteristic ODE}
\left\{\begin{aligned}
\frac{d}{dt}\Xi_c(t,0,x,p)&=\Big(v_c (\mathcal{P}_c(t,0, x,p)), E\big(t,\mathcal{X}_c(t,0,x,p)\big)\Big),\\
\Xi_c(0,0,x,p)&=(x,p).
\end{aligned}\right.
\end{equation}
Similarly, for all $t \geq 0$, we define the free flow map
$$\Phi_c^{\textup{free}}(t):=\Big(x + tv_c(p), p\Big): \mathbb{R}^3\times \mathbb{R}^3\to\mathbb{R}^3\times \mathbb{R}^3,$$
and its corresponding inverse
$$\Phi_c^{\textup{free}}(t)^{-1}= \Big(x - tv_c(p), p\Big): \mathbb{R}^3\times \mathbb{R}^3\to\mathbb{R}^3\times \mathbb{R}^3.$$

\begin{definition}[Classical finite-time wave operator]\label{def: classical finite-time wave operator}
Given a force field $E=E(t,x)$ satisfying \eqref{eq: small force field assumption} , for $1\leq c\leq\infty$ and $t\geq0$, we define the associated (forward-in-time) \textit{classical finite-time wave operator} by 
\begin{equation}\label{eq: classical finite time wave operator}
\mathcal{W}_c(t):=\Phi_c^{\textup{free}}(t)^{-1}\circ\Phi_c(t): \mathbb{R}^3\times \mathbb{R}^3\to\mathbb{R}^3\times \mathbb{R}^3.
\end{equation}
\end{definition}

By construction, each component of the wave operator can be written as below.

\begin{lemma}[Explicit formula for the classical finite-time wave operator]\label{lemma: explicit formula for the classical finite-time wave operator}
If $E(t,x)$ satisfies \eqref{eq: small force field assumption}, then for every $1 \leq c \leq \infty$, $$\mathcal{W}_c(t)=\big(\mathcal{W}_{c;1}(t),\mathcal{W}_{c;2}(t)\big):\mathbb{R}^3\times \mathbb{R}^3\to\mathbb{R}^3\times \mathbb{R}^3$$
can be expressed as
\begin{equation}\label{eq: explicit representation of the finite-time wave operator-2}
\left\{\begin{aligned}
\mathcal{W}_{c;1}(t)(x,p):&=x-\int_0^t \tau\mathbb{A}_c\big(\mathcal{P}_c(\tau,0,x,p)\big)E\big(\tau,\mathcal{X}_c(t,0,x,p)\big) d\tau,\\
\mathcal{W}_{c;2}(t)(x,p):&=p+\int_0^tE\big(\tau,\mathcal{X}_c(\tau,0,x,p)\big)d\tau,
\end{aligned}\right.
\end{equation}
where $\mathbb{A}_c(p)$ is given by \eqref{Adef}.
\end{lemma}

\begin{proof}
Fix $(x,p)\in\mathbb{R}^6$ and for brevity denote $\mathcal{W}_{c}(t)=\mathcal{W}_{c}(t)(x,p)$, $\mathcal{W}_{c;1}(t)=\mathcal{W}_{c;1}(t)(x,p)$, $\mathcal{W}_{c;2}(t)=\mathcal{W}_{c;2}(t)(x,p)$, $\mathcal{X}_c(t)=\mathcal{X}_c(t,0,x,p)$ and $\mathcal{P}_c(t)=\mathcal{P}_c(t,0,x,p)$. By definition \eqref{eq: classical finite time wave operator}, the finite-time wave operator can be written as 
$$\mathcal{W}_c(t)=\Big(\mathcal{X}_c(t)-tv_c\big(\mathcal{P}_c(t)\big),\mathcal{P}_c(t)\Big).$$
Hence, by \eqref{eq: forward characteristic ODE}, we have
$$\mathcal{W}_{c;2}(t)=\mathcal{P}_c(t)=p+\int_0^t E\big(\tau,\mathcal{X}_c(\tau)\big)d\tau.$$
Moreover, due to \eqref{eq: forward characteristic ODE} we obtain 
$$\begin{aligned}
\mathcal{W}_{c;1}(t)&=x+\int_0^t v_c\big(\mathcal{P}_c(\tau_1)\big)d\tau_1-tv_c\big(\mathcal{P}_c(t)\big)=x-\int_0^t v_c\big(\mathcal{P}_c(t)\big)-v_c\big(\mathcal{P}_c(\tau_1)\big)d\tau_1\\
&=x-\int_0^t \int_{\tau_1}^t \frac{d}{d\tau}v_c\big(\mathcal{P}_c(\tau)\big)d\tau d\tau_1=x-\int_0^t \int_{\tau_1}^t \mathbb{A}_c\big(\mathcal{P}_c(\tau)\big)E\big(\tau,\mathcal{X}_c(\tau)\big) d\tau d\tau_1\\
&=x-\int_0^t \int_0^{\tau} \mathbb{A}_c\big(\mathcal{P}_c(\tau)\big)E\big(\tau,\mathcal{X}_c(\tau)\big) d\tau_1 d\tau=x-\int_0^t \tau\mathbb{A}_c\big(\mathcal{P}_c(\tau)\big)E\big(\tau,\mathcal{X}_c(\tau)\big) d\tau,
\end{aligned}$$
where Fubini's theorem is used in the second last step.
\end{proof}

Next, due to the smallness condition on the field, the finite-time wave operator is a perturbation of the identity in the following sense.

\begin{lemma}[Almost identity]
\label{lemma:  the wave operator, almost identity}
If $E(t,x)$ satisfies \eqref{eq: small force field assumption} , then
\begin{equation}\label{eq: approx identity, finite-time wave operator}
\sup_{t\geq 0}\Big\|\mathcal{W}_c(t)(x,p)-(x,p)\Big\|_{C_{x,p}(\mathbb{R}^6)}+\sup_{t\geq 0}\Big\|\Big(\nabla_{(x,p)}\mathcal{W}_c(t)\Big)(x,p)-\mathbb{I}_6\Big\|_{C_{x,p}(\mathbb{R}^6)}\lesssim\eta_0,
\end{equation}
where the implicit constant is independent of $1 \leq c \leq \infty$. Thus, if $\eta_0>0$ is sufficiently small, then $\mathcal{W}_c(t)$ is invertible with 
$$\mathcal{W}_c(t)^{-1} = \Phi_c(t)^{-1} \circ \Phi_c^{\textup{free}}(t)$$
and
\begin{equation}\label{eq: approx identity, inverse finite-time wave operator}
\sup_{t\geq 0}\Big\|\mathcal{W}_c(t)^{-1}(x,p)-(x,p)\Big\|_{C_{x,p}(\mathbb{R}^6)}+\sup_{t\geq 0}\Big\|\Big(\nabla_{(x,p)}\mathcal{W}_c(t)^{-1}\Big)(x,p)-\mathbb{I}_6\Big\|_{C_{x,p}(\mathbb{R}^6)}\lesssim\eta_0.
\end{equation}
\end{lemma} 

\begin{proof}
We denote $\mathcal{X}_c(t)=\mathcal{X}_c(t,0,x,p)$ and $\mathcal{P}_c(t)=\mathcal{P}_c(t,0,x,p)$ fixing $(x,p)\in\mathbb{R}^6$.  Then, 
%as in the proof of Lemma \ref{lemma: explicit formula for the classical wave operator}, 
we write 
$$\big(\mathcal{Y}_1(t), \mathcal{Y}_2(t)\big):=\mathcal{W}_c(t)(x,p)=\Big(\mathcal{X}_c(t)-tv_c\big(\mathcal{P}_c(t)\big), \mathcal{P}_c(t)\Big),$$
where $\mathcal{Y}_1=(\mathcal{Y}_{1;1}, \mathcal{Y}_{1;2}, \mathcal{Y}_{1;3})$ and $\mathcal{Y}_2=(\mathcal{Y}_{2;1}, \mathcal{Y}_{2;2}, \mathcal{Y}_{2;3})$. Then, applying \eqref{eq: Ac asymptotic'} and \eqref{eq: small force field assumption} to the representation \eqref{eq: explicit representation of the finite-time wave operator-2}, we obtain 
$$\begin{aligned}
\big|\big(\mathcal{Y}_1(t), \mathcal{Y}_2(t)\big)-(x,p)\big|&\leq\int_0^t (1+\tau)\|E(\tau)\|_{L_x^\infty} d\tau\lesssim \int_0^t (1+\tau)\frac{\eta_0}{(1+\tau)^{\alpha+1}} d\tau\lesssim\eta_0.
\end{aligned}$$

On the other hand, the derivatives are given by 
$$\begin{aligned}
\partial_{x_j}\mathcal{Y}_{1;k}(t)&=\delta_{jk}-\sum_{\ell=1}^3\int_0^t \tau\mathbb{A}_c^{k\ell}\big(\mathcal{P}_c(\tau)\big) \nabla_x E_\ell\big(\tau,\mathcal{X}_c(\tau)\big)\cdot\partial_{x_j}\mathcal{X}_c(\tau)d\tau\\
&\quad-\int_0^t  \tau\nabla\mathbb{A}_c^{k\ell}\big(\mathcal{P}_c(\tau)\big)\cdot\partial_{x_j}\mathcal{P}_c(\tau) E_\ell\big(\tau,\mathcal{X}_c(\tau)\big)d\tau,\\
\partial_{p_j}\mathcal{Y}_{1;k}(t)&=-\sum_{\ell=1}^3\int_0^t \tau\mathbb{A}_c^{k\ell}\big(\mathcal{P}_c(\tau)\big) \nabla_x E_\ell\big(\tau,\mathcal{X}_c(\tau)\big)\cdot\partial_{p_j}\mathcal{X}_c(\tau)d\tau\\
&\quad-\int_0^t  \tau\nabla\mathbb{A}_c^{k\ell}\big(\mathcal{P}_c(\tau)\big)\cdot\partial_{p_j}\mathcal{P}_c(\tau) E_\ell\big(\tau,\mathcal{X}_c(\tau)\big)d\tau,\\
\partial_{x_j}\mathcal{Y}_{2;k}(t)&=\int_0^t  \nabla_x E_k\big(\tau,\mathcal{X}_c(\tau)\big)\cdot\partial_{x_j}\mathcal{X}_c(\tau)d\tau\\
\partial_{p_j}\mathcal{Y}_{2;k}(t)&=\delta_{jk}+\int_0^t  \nabla_x E_k\big(\tau,\mathcal{X}_c(\tau)\big)\cdot\partial_{p_j}\mathcal{X}_c(\tau)d\tau,
\end{aligned}$$
where $\mathbb{A}_c=(\mathbb{A}_c^{j\ell})_{j,\ell}$ and $E=(E_1, E_2, E_3)$. 
Hence, due to Lemma \ref{lemma: bounds for the relativistic velocity and its derivatives} and \eqref{eq: Ac difference bound} with
$$(\mathcal{X}_c(t), \mathcal{P}_c(t))=\Big(\mathcal{Y}_1(t)+tv_c\big(\mathcal{Y}_2(t)\big), \mathcal{Y}_2(t)\Big)$$
we find
$$\begin{aligned}
\|\nabla_x\mathcal{Y}_1(t)-\mathbb{I}_3\|&\lesssim\int_0^t \tau\Big\{\|\nabla_x E(\tau)\|_{L_x^\infty}\|\nabla_x\mathcal{X}_c(\tau)\|+\|\nabla_x\mathcal{P}_c(\tau)\|\| E(\tau)\|_{L_x^\infty}\Big\}d\tau,\\
&\lesssim\eta_0+\int_0^t \left ( \frac{\eta_0}{(1+\tau)^{\alpha+1}}\|\nabla_x\mathcal{Y}_1(\tau)-\mathbb{I}_3\|+\frac{\eta_0}{(1+\tau)^\alpha}\|\nabla_x\mathcal{Y}_2(\tau)\| \right ) d\tau,
\end{aligned}$$
and similarly, 
$$\begin{aligned}
\|\nabla_p\mathcal{Y}_1(t)\|&\lesssim\eta_0+\int_0^t \left ( \frac{\eta_0}{(1+\tau)^{\alpha+1}}\|\nabla_p\mathcal{Y}_1(\tau)\|+\frac{\eta_0}{(1+\tau)^{\alpha}}\|\nabla_p\mathcal{Y}_2(s)-\mathbb{I}_3\| \right ) d\tau,\\
\|\nabla_x\mathcal{Y}_2(t)\|&\lesssim\eta_0+\int_0^t \left ( \frac{\eta_0}{(1+\tau)^{\alpha+2}}\|\nabla_x\mathcal{Y}_1(\tau)-\mathbb{I}_3\|+\frac{\eta_0}{(1+\tau)^{\alpha+1}}\|\nabla_x\mathcal{Y}_2(\tau)\| \right ) d\tau,\\
\|\nabla_p\mathcal{Y}_2(t)-\mathbb{I}_3\|&\lesssim\eta_0+\int_0^t  \left ( \frac{\eta_0}{(1+\tau)^{\alpha+2}}\|\nabla_p\mathcal{Y}_1(\tau)\|+\frac{\eta_0}{(1+\tau)^{\alpha+1}}\|\nabla_p\mathcal{Y}_2(\tau)-\mathbb{I}_3\| \right ) d\tau.
\end{aligned}$$
Adding these estimates and applying Gr\"onwall's inequality with $\alpha > 1$, we find 
$$\Big\|\nabla_{(x,p)}\Big(\mathcal{Y}_1(t),\mathcal{Y}_2(t)\Big)-\mathbb{I}_6\Big\|\lesssim\eta_0,$$ 
which proves the bound for $\nabla_{(x,p)}\mathcal{W}_c(t)(x,p)-\mathbb{I}_6$. Then, the properties of the inverse $\mathcal{W}_c(t)^{-1}$ follow from \eqref{eq: approx identity, finite-time wave operator} and implicit differentiation. 
\end{proof}

\subsection{Limiting classical wave operator}\label{sec: Limiting classical wave operator}

Applying \eqref{eq: Ac asymptotic'} and \eqref{eq: small force field assumption} to \eqref{eq: explicit representation of the finite-time wave operator-2}, we observe
$$\begin{aligned}
\big|\mathcal{W}_{c;1}(t_2)(x,p)-\mathcal{W}_{c;1}(t_1)(x,p)\big|&\leq \int_{t_1}^{t_2} \tau \|E(\tau)\|_{L_x^\infty} d\tau\lesssim \int_{t_1}^{t_2} \frac{\eta_0}{(1+\tau)^\alpha} d\tau\to 0,\\
\big|\mathcal{W}_{c;2}(t_2)(x,p)-\mathcal{W}_{c;2}(t_1)(x,p)\big|&\leq \int_{t_1}^{t_2} \|E(\tau)\|_{L_x^\infty} d\tau\lesssim \int_{t_1}^{t_2} \frac{\eta_0}{(1+\tau)^{\alpha+1}} d\tau\to 0
\end{aligned}$$
as $t_2\geq t_1\to\infty$. Therefore, for each $(x,p)$, the limit of $\mathcal{W}(t)(x,p)$ exists as $t\to\infty$. 

\begin{definition}[Classical wave operator]
Under the assumption \eqref{eq: small force field assumption}  on $E=E(t,x)$, the  (forward-in-time) \textit{classical wave operator} $\mathcal{W}_c^+=\mathcal{W}_c^{E; +}$ is defined by 
\begin{equation}\label{eq: classical wave operator 1}
\mathcal{W}_c^+=\big(\mathcal{W}_{c;1}^+, \mathcal{W}_{c;2}^+\big):=\lim_{t\to+\infty}\mathcal{W}_c(t):\mathbb{R}^6\to\mathbb{R}^6,
\end{equation}
for every $1 \leq c \leq \infty$, where $\mathcal{W}_c(t)$ is given in Definition \ref{def: classical finite-time wave operator}. 
\end{definition}

\begin{remark}
The classical wave operator $\mathcal{W}_c^+$ is explicitly defined in the short-range interaction case. By construction, it preserves volume.
\end{remark}

\begin{lemma}[Properties of the classical wave operator]\label{lemma: explicit formula for the classical wave operator}
Suppose that $E(t,x)$ satisfies \eqref{eq: small force field assumption}. Then, for every $1 \leq c \leq \infty$, the wave operator  
$$\mathcal{W}_c^+(x,p)=\Big(\mathcal{X}_c^+(x,p),\mathcal{P}_c^+(x,p)\Big):\mathbb{R}^6\to\mathbb{R}^6$$
is given by
\begin{equation}\label{eq: explicit representation of the wave operator-2}
\left\{\begin{aligned}
\mathcal{X}_c^+(x,p):&=x-\int_0^\infty t\mathbb{A}_c\big(\mathcal{P}_c(t,0,x,p)\big)E_c\big(t,\mathcal{X}_c(t,0,x,p)\big) dt,\\
\mathcal{P}_c^+(x,p):&=p+\int_0^\infty E_c\big(t,\mathcal{X}_c(t,0,x,p)\big)dt.
\end{aligned}\right.
\end{equation}
Similar to Lemma \ref{lemma:  the wave operator, almost identity}, it satisfies
\begin{equation}\label{eq: approx identity, wave operator}
\Big\|\mathcal{W}_c^+(x,p)-(x,p)\Big\|_{C_{x,p}(\mathbb{R}^6)}+\Big\|\nabla_{(x,p)}\mathcal{W}_c^+(x,p)-\mathbb{I}_6\Big\|_{C_{x,p}(\mathbb{R}^6)}\lesssim\eta_0,
\end{equation}
\begin{equation}\label{eq: approx identity, inverse wave operator}
\Big\|\mathcal{W}_c^+(x,p)-(x,p)\Big\|_{C_{x,p}(\mathbb{R}^6)}+\Big\|\Big(\nabla_{(x,p)}(\mathcal{W}_c^+)^{-1}\Big)(x,p)-\mathbb{I}_6\Big\|_{C_{x,p}(\mathbb{R}^6)}\lesssim\eta_0.
\end{equation}
Moreover, for $t\geq 0$, it satisfies the convergence estimate 
\begin{equation}\label{eq: rate of convergence to the wave operator}
\sup_{(x,p)\in\mathbb{R}^6}\big|\mathcal{W}_c(t)(x,p)-\mathcal{W}_c^+(x,p)\big|\lesssim (1+t)^{-\alpha}.
\end{equation}
where the implicit constant is independent of $1 \leq c \leq \infty$ and $t\geq 0$. 
\end{lemma}

\begin{proof}
The proofs of \eqref{eq: explicit representation of the wave operator-2}, \eqref{eq: approx identity, wave operator}, and \eqref{eq: approx identity, inverse wave operator} closely follow from those of \eqref{eq: explicit representation of the finite-time wave operator-2}, \eqref{eq: approx identity, finite-time wave operator}, and \eqref{eq: approx identity, inverse finite-time wave operator} except taking $t=\infty$. Therefore, we only show  \eqref{eq: rate of convergence to the wave operator}. Indeed, comparing the integral representations \eqref{eq: explicit representation of the wave operator-2} and \eqref{eq: explicit representation of the finite-time wave operator-2}, the difference can be written as 
$$\left\{\begin{aligned}
\mathcal{W}_{c;1}(t)(x,p)-\mathcal{W}_{c;1}^+(x,p)&=\int_t^\infty  \tau\mathbb{A}_c\big(\mathcal{P}_c(\tau)\big) E_c\big(\tau,\mathcal{X}_c(\tau,0,x,p)\big)d\tau,\\
\mathcal{W}_{c;2}(t)(x,p)-\mathcal{W}_{c;2}^+(x,p)&=-\int_t^\infty E_c\big(\tau,\mathcal{X}_c(\tau,0,x,p)\big)d\tau.
\end{aligned}\right.$$
Then, \eqref{eq: rate of convergence to the wave operator} follows directly from \eqref{eq: small force field assumption} and Lemma \ref{lemma: bounds for the relativistic velocity and its derivatives}.
\end{proof}

By definition, the wave operator also possesses the intertwining property.

\begin{lemma}[Intertwining property of the classical wave operator]\label{lemma: intertwining property of the classical wave operator}
If \eqref{eq: small force field assumption}  holds, then
for every $1 \leq c \leq \infty$ and $t \geq 0$, we have
$$\mathcal{W}_c^+\circ\Phi_c(t)=\Phi_c^{\textup{free}}(t)\circ \mathcal{W}_c^+.$$
\end{lemma}

\begin{proof}
As defined, both $\Phi_c(t)$ and $\Phi_c^{\textup{free}}(t)$ are one-parameter groups. Thus, we have
$$\begin{aligned}
\mathcal{W}_c(T)\circ\Phi_c(t) &=
\big(\Phi_c^{\textup{free}}(T)^{-1}\circ\Phi_c(T)\big)\circ\Phi_c(t)\\
&=\Phi_c^{\textup{free}}(T)^{-1}\circ\Phi_c(T+t)\\
&=\Phi_c^{\textup{free}}(t)\circ \Phi_c^{\textup{free}}(T+t)^{-1}\circ\Phi_c(T+t)\\
&=\Phi_c^{\textup{free}}(t)\circ \mathcal{W}_c(T+t).
\end{aligned}$$
Hence, taking $T\to\infty$, we obtain the stated result.
\end{proof}

\subsection{Density function estimates}

As an application of the wave operator, we prove the following density function bounds for perturbed flows. An important remark is that all estimates below hold uniformly for $1 \leq c \leq \infty$.

\begin{proposition}[Density function estimates for perturbed relativistic flows]\label{prop: density function estimates for perturbed relativistic flows}
Under the assumption \eqref{eq: small force field assumption}  with
$$f_c(t,x,p):=f^0\Big(\Phi_c(t)^{-1}(x,p)\Big)$$
for $f^0=f^0(x,p)\geq 0$, we have 
\begin{equation}\label{eq: L 1 density function estimate}
\|\rho_{f_c}(t)\|_{L_x^1(\mathbb{R}^3)}=\|f^0\|_{L_{x,p}^1(\mathbb{R}^6)}
\end{equation}
and
\begin{equation}\label{eq: L infinity density function estimate}
\|\rho_{f_c}(t)\|_{L_x^\infty(\mathbb{R}^3)}\lesssim\frac{1}{(1+t)^3}\Big\|\langle x\rangle^{3^+}\langle p\rangle^5 f^0\Big\|_{L_{x,p}^\infty(\mathbb{R}^6)}.
\end{equation}
\end{proposition}

\begin{remark}
The proof of Proposition \ref{prop: density function estimates for perturbed relativistic flows} is divided into two parts. A decay bound is obtained from the free flow $(x,p)\mapsto (x+tv_c(p),p)$. Then, boundedness of the wave operator is used. A similar approach is employed in the proof of Proposition \ref{prop: derivative bounds for perturbed relativistic flows}.
\end{remark}

\begin{proof}
The first inequality \eqref{eq: L 1 density function estimate} is trivial as the volume preserving property of $\Phi_c(t)$ implies
$$\|\rho_{f_c}(t)\|_{L_x^1}=\iint_{\mathbb{R}^6}f^0\Big(\Phi_c(t)^{-1}(x,p)\Big)dxdp=\|f^0\|_{L_{x,p}^1}.$$
For \eqref{eq: L infinity density function estimate}, using the notation in \eqref{eq: T[g]}, one can write 
\begin{equation}\label{eq: rho_c T1 expression}
\rho_{f_c}(t,x) 
= \int_{\mathbb{R}^3} g_c \Big(t, x-tv_c(p),p \Big) dp=T^1[g_c](t,x),
\end{equation}
where 
\begin{equation}\label{eq: gc definition0}
g_c(t,x,p):=f_c\Big(t,x+tv_c(p),p\Big)=f^0\Big(\mathcal{W}_c(t)^{-1}(x,p)\Big).
\end{equation}
Then, it follows from Lemma \ref{lemma: dispersive bounds for the free flow associated with interpolated v} that 
$$\begin{aligned}
\|\rho_{f_c}(t)\|_{L_x^\infty}&\lesssim\frac{1}{(1+t)^3}\Big\|\Big(\langle x\rangle^{3^+}\langle p\rangle^{5}\Big)f^0\Big(\mathcal{W}_c(t)^{-1}(x,p)\Big)\Big\|_{L_{x,p}^\infty}\\
&=\frac{1}{(1+t)^3}\Big\|\Big(\langle x(\tilde{x},\tilde{p})\rangle^{3^+}\langle p(\tilde{x},\tilde{p})\rangle^{5}\Big)f^0(\tilde{x},\tilde{p})\Big\|_{L_{\tilde{x},\tilde{p}}^\infty},
\end{aligned}$$
where $(x,p)=(x(\tilde{x},\tilde{p}),p(\tilde{x},\tilde{p}))=\mathcal{W}_c(\tilde{x},\tilde{p})$. However, because the wave operator is a perturbation of the identity by Lemma \ref{lemma:  the wave operator, almost identity}, we have $|(x(\tilde{x},\tilde{p}),p(\tilde{x},\tilde{p}))-(\tilde{x},\tilde{p})|\lesssim\eta_0$ so that
$\langle x(\tilde{x},\tilde{p})\rangle\sim\langle \tilde{x}\rangle $ and $\langle p(\tilde{x},\tilde{p})\rangle\sim\langle \tilde{p}\rangle$, 
and \eqref{eq: L infinity density function estimate} follows.
\end{proof}

Next, we prove the bounds for derivatives. In particular, we show that derivatives of the density function decay faster than the density itself.

\begin{proposition}[Derivative bounds for perturbed relativistic flows]
\label{prop: derivative bounds for perturbed relativistic flows}  Under the assumptions of Proposition \ref{prop: density function estimates for perturbed relativistic flows}, we have
\begin{equation}
\label{eq: derivative bound without extra decay}
\|\nabla_x\rho_{f_c}(t)\|_{L_x^1(\mathbb{R}^3)} \leq \|\nabla_{(x,p)}f^0\|_{L_{x,p}^1(\mathbb{R}^6)},
\end{equation}
\begin{equation}
\label{eq: derivative bound with extra decay}
\|\nabla_x\rho_{f_c}(t)\|_{L_x^1(\mathbb{R}^3)} \leq
\frac{1}{t}\left (\| \langle p \rangle^3 \nabla_{(x,p)} f^0 \|_{L^1_{x,p}(\mathbb{R}^6)} + \frac{1}{c^2} \| \langle p \rangle^2 f^0 \|_{L^1_{x,p}(\mathbb{R}^6)} \right ),
\end{equation}
and
\begin{equation}\label{eq: L infinity derivative bound}
\begin{aligned}
\|\nabla_x\rho_{f_c}(t)\|_{L_x^\infty(\mathbb{R}^3)}&\lesssim\frac{1}{(1+t)^4}\Big\|\langle x\rangle^{3^+}\langle p\rangle^{8}\nabla_{(x,p)}f^0\Big\|_{L_{x,p}^\infty(\mathbb{R}^3)}\\
&\quad+\frac{1}{(1+t)^4c^2} \Big\|\langle x\rangle^{3^+}\langle p\rangle^{7} f^0\Big\|_{L_{x,p}^\infty(\mathbb{R}^3)}.
\end{aligned}\end{equation}
\end{proposition}

\begin{proof}
Differentiating \eqref{eq: rho_c T1 expression} with \eqref{eq: gc definition0}, we write 
$$\begin{aligned}
\nabla_x\rho_{f_c}(t,x) 
&= \int_{\mathbb{R}^3} (\nabla_xg_c) \Big(t, x-tv_c(p),p \Big) dp\\
&= \int_{\mathbb{R}^3} (\nabla_{(x,p)}f^0) \Big(\mathcal{W}_c(t)^{-1}(x-tv_c(p),p) \Big)\cdot\nabla_x\Big(\mathcal{W}_c(t)^{-1}(x-tv_c(p),p)\Big) dp.
\end{aligned}$$
Note that by Lemma \ref{lemma:  the wave operator, almost identity}, we have
$$\sup_{t\geq 0}\Big\|\Big(\nabla_{(x,p)}\mathcal{W}_c(t)^{-1}\Big)(x,p)-\mathbb{I}_6\Big\|_{C_{x,p}(\mathbb{R}^6)}\lesssim\eta_0,$$
and in particular, $\|\nabla_x(\mathcal{W}_c(t)^{-1})(x,p)-(\mathbb{I}_3,0)\|\lesssim\eta_0$.  Hence, using this within the expression for the gradient of $\rho_{f_c}$, we obtain the bound 
$$\begin{aligned}
|\nabla_x\rho_{f_c}(t,x)|&\lesssim \int_{\mathbb{R}^3}\Big|\nabla_{(x,p)}f^0\Big(\mathcal{W}_c(t)^{-1}(x-tv_c(p),p)\Big) \Big| dp\\
&=\int_{\mathbb{R}^3}\Big|\nabla_{(x,p)}f^0\Big(\Phi_c(t)^{-1}(x,p)\Big) \Big|dp=\rho_{(|\nabla_{(x,p)}f^0|)(\Phi_c(t)^{-1}(x,p))}.
\end{aligned}$$
Then, as the right side is merely the momentum average of $(|\nabla_{(x,p)}f^0|)(\Phi_c(t)^{-1}(x,p))$, the previously-established estimates \eqref{eq: L 1 density function estimate} and \eqref{eq: L infinity density function estimate} yield  \eqref{eq: derivative bound without extra decay} and
\begin{equation}\label{eq: L infinity derivative bound without extra decay}
\|\nabla_x\rho_{f_c}(t)\|_{L_x^\infty}\lesssim\frac{1}{(1+t)^3}\Big\|\langle x\rangle^{3^+}\langle p\rangle^5\nabla_{(x,p)}f^0(x,p)\Big\|_{L_{x,p}^\infty}.
\end{equation}

It remains to show the faster $t^{-4}$ decay bound in \eqref{eq: L infinity derivative bound}. For this, contrary to the proof of \eqref{eq: derivative bound without extra decay}, we change variables in the expression \eqref{eq: rho_c T1 expression} first before differentiating. Precisely, changing variables by
$$p\mapsto z=x-tv_c(p):\mathbb{R}^3\to B(x,ct),$$
or equivalently $p=v_c^{-1}\left ( \frac{x-z}{t} \right )$ with $|\frac{\partial z}{\partial p}|=\frac{t^3}{\gamma_c(p)^5}$ (see Lemma \ref{lemma: Ac eigenvalues} with $\theta=1$), we obtain
\begin{equation}\label{eq: fc wave operator formulation, density}
\rho_{f_c}(t,x)=\frac{1}{t^3}\int_{B(x,ct)} g_c\bigg(t, z,v_c^{-1}\bigg(\frac{x-z}{t}\bigg)\bigg)\gamma_c \bigg(v_c^{-1}\bigg(\frac{x-z}{t}\bigg)\bigg)^5 dz.
\end{equation}
Thus, its derivative is given by
$$\begin{aligned}
\partial_{x_j}\rho_{f_c}(t,x)&=\frac{1}{t^3}\int_{B(x,ct)} \biggl(\nabla_p g_c(t,z, p)\cdot \partial_{x_j}p \  \gamma_c(p)^5 \biggr) \biggl |_{p=v_c^{-1} \left (\frac{x-z}{t} \right)} dz\\
&\quad+\frac{5}{c^2 t^3}\int_{B(x,ct)} \biggl (g_c(t,z, p) \gamma_c(p)^3 p \cdot \partial_{x_j} p \biggr ) \biggl |_{p=v_c^{-1} \left (\frac{x-z}{t} \right)}  dz,
\end{aligned}$$
where the boundary term vanishes because $v_c^{-1} \left (\frac{x-z}{t} \right) \to \infty$ as $|x-z|\to ct$. Note that $$\partial_{x_j}p_k=\frac{\delta_{jk}}{t(1-|\frac{x-z}{ct}|^2)^{1/2}}+\frac{\frac{(x_j-z_j)(x_k-z_k)}{c^2t^2}}{t(1-|\frac{x-z}{ct}|^2)^{3/2}}=\frac{1}{t} \gamma_c(p) \bigg(\delta_{jk}+\frac{p_jp_k}{c^2}\bigg).$$
Hence, it follows that 
$$\begin{aligned}
|\partial_{x_j}\rho_{f_c}(t,x)|&\lesssim\frac{1}{t^4}\int_{B(x,ct)} \biggl ( |\nabla_p g_c(t,z,p)| \gamma_c(p)^8 \bigg) \biggl |_{p=v_c^{-1} \left (\frac{x-z}{t} \right)} dz\\
&\quad+\frac{1}{t^4c^2}\int_{B(x,ct)} \biggl (g_c(t,z,p)  \gamma_c(p)^6 |p| \biggr ) \biggl |_{p=v_c^{-1} \left (\frac{x-z}{t} \right)} dz.
\end{aligned}$$
For the right hand side, because $g_c(t,x,p)=f^0(\mathcal{W}_c(t)^{-1}(x,p))$, Lemma \ref{lemma:  the wave operator, almost identity} and \eqref{eq: bound for Pc} imply
$$\begin{aligned}
|\nabla_p g_c(t,x,p)| \gamma_c(p)^8&=|\nabla_{(x,p)}f^0(\mathcal{W}_c(t)^{-1}(x,p))\cdot \nabla_p(\mathcal{W}_c(t)^{-1})(x,p)|\gamma_c(p)^8\\
&\lesssim \Big|\Big(\gamma_c(p)^3\nabla_{(x,p)}f^0\Big)(\mathcal{W}_c(t)^{-1}(x,p))\Big|\gamma_c(p)^5
\end{aligned}$$
and similarly
$$\begin{aligned}
|g_c(t,x,p)|  \gamma_c(p)^6 |p|\lesssim \Big|\Big(\gamma_c(p) |p|f^0\Big)(\mathcal{W}_c(t)^{-1}(x,p))\Big|\gamma_c(p)^5.
\end{aligned}$$
Therefore, the derivative estimate becomes
$$\begin{aligned}
|\partial_{x_j}\rho_{f_c}(t,x)|&\lesssim\frac{1}{t^4}\int_{B(x,ct)} \biggl ( \Big|\Big(\gamma_c(p)^3\nabla_{(x,p)}f^0\Big)(\mathcal{W}_c(t)^{-1}(z,p))\Big|\gamma_c(p)^5 \bigg) \biggl |_{p=v_c^{-1} \left (\frac{x-z}{t} \right)} dz\\
&\quad+\frac{1}{t^4c^2}\int_{B(x,ct)} \biggl (\Big|\Big(\gamma_c(p) |p|f^0\Big)(\mathcal{W}_c(t)^{-1}(z,p))\Big|\gamma_c(p)^5\biggr ) \biggl |_{p=v_c^{-1} \left (\frac{x-z}{t} \right)} dz\\
&=\frac{1}{t}\rho_{(|\gamma_c(p)^3\nabla_{(x,p)}f^0|)(\Phi_c(t)^{-1}(x,p))}+\frac{1}{tc^2}\rho_{(|\gamma_c(p)|p|f^0|)(\Phi_c(t)^{-1}(x,p))},
\end{aligned}$$
where \eqref{eq: fc wave operator formulation, density} is used conversely in the last step. Therefore, applying the dispersion estimate \eqref{eq: L infinity density function estimate}, we conclude
$$\begin{aligned}
\|\nabla_x\rho_{f_c}(t)\|_{L_x^\infty}&\lesssim\frac{1}{t(1+t)^3}\Big\|\langle x\rangle^{3^+}\langle p\rangle^{5}\gamma_c(p)^3\nabla_{(x,p)}f^0(x,p)\Big\|_{L_{x,p}^\infty}\\
&\quad+\frac{1}{t(1+t)^3c^2} \Big\|\langle x\rangle^{3^+}\langle p\rangle^5\gamma_c(p)pf^0(x,p)\Big\|_{L_{x,p}^\infty}.
\end{aligned}$$
Finally, combining with \eqref{eq: L infinity derivative bound without extra decay} to rule out the singularity at $t=0$ in the upper bound of the above inequality and using $\gamma_c(p) \leq \langle p \rangle$, we obtain the desired bound \eqref{eq: L infinity derivative bound}.
Similarly, estimating $\| \nabla \rho_{f_c}(t) \|_{L^1_x}$ by integrating the above estimate on $|\partial_{x_j}\rho_{f_c}(t,x)|$ yields 
$$\begin{aligned}
\|\nabla_x\rho_{f_c}(t)\|_{L_x^1(\mathbb{R}^3)} &\lesssim t^{-1}\left ( \| \gamma_c(p)^3 \nabla_{(x,p)} f^0 \|_{L^1_{x,p}} + c^{-2} \| \gamma_c(p) p f^0 \|_{L^1_{x,p}} \right )\\
&\lesssim t^{-1}\left ( \| \langle p \rangle^3 \nabla_{(x,p)} f^0 \|_{L^1_{x,p}} + c^{-2} \| \langle p \rangle^2 f^0 \|_{L^1_{x,p}} \right )
\end{aligned}$$
by using the volume-preserving property of $\Phi_c(t)^{-1}$.
\end{proof}

\section{Uniform bounds and scattering for the relativistic Vlasov equation}\label{sec: {Uniform bounds and scattering for the relativistic Vlasov equation}}

Within this section, we assume that the initial data satisfies the condition that
\begin{equation}\label{eq: initial data assumption2}
\eta = \Big\|\langle x\rangle^{3^+}\langle p \rangle^8 f^0\Big\|_{L_{x,p}^\infty} +\Big\|\langle x\rangle^{3^+}\langle p \rangle^9\nabla_{(x,p)}f^0\Big\|_{L_{x,p}^\infty}
\end{equation}
is sufficiently small and use this to construct global-in-time solutions.
Subsequently, we obtain the large-time behavior of these solutions and show that they scatter along the forward free flow as $t \to \infty$.

\subsection{Global existence and uniqueness for the Vlasov equation}
\label{GE}

\begin{proof}[Proof of Theorem \ref{T1}]
We construct a sequence $\{E_c^{(j)}\}_{j=1}^\infty$ of force fields and a sequence
$$\big\{(\Xi_c^{(j)}(s,t, x,p)\big\}_{j=1}^\infty=\big\{\big(\mathcal{X}_c^{(j)}(s,t, x,p), \mathcal{P}_c^{(j)}(s,t, x,p)\big)\big\}_{j=1}^\infty$$
of characteristics as follows. First, for $n=1$, we set $E_c^{(1)}(t,x)\equiv0$.  Then, for any $j\geq1$, let $\Xi_c^{(j)}(s,t, x,p)$ be the solution to the characteristic ODE \eqref{eq: r-ODE} with $E=E_c^{(j)}$, and define
 $$E_c^{(j+1)}(t,x):=\beta \big(\nabla w*\rho_{f_c^{(j)}}\big)(t,x)=\beta\iint_{\mathbb{R}^3\times\mathbb{R}^3}\nabla w(x-y)f_c^{(j)}(t,y,p)dydp,$$
 where
 $$f_c^{(j)}(t,x,p):=f^0\big(\Xi_c^{(j)}(0,t, x,p)\big).$$
 We claim that
\begin{equation}\label{eq: iterative force field bound assumption}
\sup_{t\geq0}\Big\{(1+t)^{\alpha+1}\|E_c^{(j)}(t)\|_{L_x^\infty}+(1+t)^{\alpha+2}\|\nabla_x E_c^{(j)}(t)\|_{L_x^\infty}\Big\}\leq\eta_0,
\end{equation}
where $\eta_0$ is a small constant in \eqref{eq: small force field assumption} satisfying $0 < \eta < \eta_0 \ll 1$. Indeed, it follows immediately for $j=1$. Furthermore, suppose that \eqref{eq: iterative force field bound assumption} holds for some $j\geq 1$. Then, by Propositions \ref{prop: density function estimates for perturbed relativistic flows} and \ref{prop: derivative bounds for perturbed relativistic flows}, it follows that 
$$\begin{aligned}
\sup_{t\geq0}\Big\{(1+t)^3\|\rho_{f_c^{(j)}}(t)\|_{L_x^\infty}+(1+t)^4\|\nabla_x \rho_{f_c^{(j)}}(t)\|_{L_x^\infty}\Big\} &\leq\eta,\\
\sup_{t\geq0}\Big\{\|\rho_{f_c^{(j)}}(t)\|_{L_x^1}+(1+t)\|\nabla_x \rho_{f_c^{(j)}}(t)\|_{L_x^1}\Big\} &\leq\eta.
\end{aligned}$$
Now, Lemma \ref{lemma: interpolation inequality} implies
$$\begin{aligned}
\|E_c^{(j+1)}(t)\|_{L_x^\infty}&= \|\beta\nabla w*\rho_{f_c^{(j)}}(t)\|_{L_x^\infty}\lesssim \|\rho_{f_c^{(j)}}(t)\|_{L_x^1}^{\frac{2-\alpha}{3}}\|\rho_{f_c^{(j)}}(t)\|_{L_x^\infty}^{\frac{\alpha+1}{3}}\lesssim\frac{\eta}{(1+t)^{\alpha+1}},\\
\|\partial_{x_k} E_c^{(j+1)}(t)\|_{L_x^\infty}&= \|\beta\nabla w*\partial_{x_k}\rho_{f_c^{(j)}}(t)\|_{L_x^\infty}\lesssim \|\partial_{x_k}\rho_{f_c^{(j)}}(t)\|_{L_x^1}^{\frac{2-\alpha}{3}}\|\partial_{x_k}\rho_{f_c^{(j)}}(t)\|_{L_x^\infty}^{\frac{\alpha+1}{3}}\lesssim\frac{\eta}{(1+t)^{\alpha+2}}
\end{aligned}$$
for any $k=1, 2, 3$.
Taking $\eta$ sufficiently small with $0 < \eta < \eta_0$ so that the constant on the right sides of these estimates is no more than $\frac{1}{2}\eta_0$ yields \eqref{eq: iterative force field bound assumption} for $E_c^{(j+1)}(t,x)$.
Hence, by induction, \eqref{eq: iterative force field bound assumption} holds for all $j \geq 1$.

Next, we show that $\{E_c^{(j)}\}_{j=1}^\infty$ and $\{(\mathcal{X}_c^{(j)}(s,t, x,p), \mathcal{P}_c^{(j)}(s,t, x,p))\}_{j=1}^\infty$ are contractive.
To this end, define
$$\delta \mathcal{X}_c^{(j+1)}(s,t,x,p) = \left | \mathcal{X}_c^{(j+1)}(s,t,x,p) - \mathcal{X}_c^{(j)} (s,t,x,p) \right |,$$
$$\delta \mathcal{P}_c^{(j+1)}(s,t,x,p) = \left | \mathcal{P}_c^{(j+1)}(s,t,x,p) - \mathcal{P}_c^{(j)} (s,t,x,p) \right |,$$
$$\delta E_c^{(j+1)}(t,x) = \left | E_c^{(j+1)}(t, x) - E_c^{(j)}(t, x)) \right |,$$
and
$$\delta f_c^{(j+1)}(t,x,p) = \left | f_c^{(j+1)}(t,x,p) - f_c^{(j)}(t,x,p) \right |.$$
By \eqref{eq: Ac asymptotic} and \eqref{eq: integral r-ODE}, 
we immediately find
$$ \delta \mathcal{X}_c^{(j+1)}(s,t,x,p) \leq
\int_s^t \delta \mathcal{P}_c^{(j+1)}(\tau,t,x,p)\ d\tau,$$
and further adding and subtracting field terms and employing \eqref{eq: iterative force field bound assumption}, we find
$$\begin{aligned}
\delta \mathcal{P}_c^{(j+1)}(s,t,x,p)
& \leq 
\int_s^t \delta E_c^{(j+1)}\big(\tau, \mathcal{X}_c^{(j+1)} (\tau,t,x,p)\big) d\tau\\
& \quad + \int_s^t \left | E_c^{(j)}\big(\tau, \mathcal{X}_c^{(j+1)} (\tau,t,x,p)\big) - E_c^{(j)}\big(\tau, \mathcal{X}_c^{(j)} (\tau,t,x,p)\big) \right | d\tau\\
& \leq \int_s^t \delta E_c^{(j+1)}\big(\tau, \mathcal{X}_c^{(j+1)} (\tau,t,x,p)\big) d\tau + 2\eta_0  \left ( \int_0^\infty (1 + \tau)^{-\alpha - 1} d\tau \right )\\
& \lesssim \eta_0 + \int_s^t \delta E_c^{(j+1)}\big(\tau, \mathcal{X}_c^{(j)} (\tau,t,x,p)\big) d\tau.
\end{aligned}$$
Due to the structure of the field, we estimate similar to Lemma \ref{lemma: interpolation inequality} to find
$$\begin{aligned}
\delta E_c^{(j+1)}(t, x) & \lesssim \iint_{|x-y| < 1} \frac{\delta f_c^{(j)}(t,y,p)}{|x-y|^{\alpha+1}} dp dy + \iint_{|x-y| > 1} \frac{\delta f_c^{(j)}(t,y,p)}{|x-y|^{\alpha+1}} dp dy\\
& \lesssim \Vert \langle p \rangle^{3^+} \delta f_c^{(j)}(t) \Vert_{L^\infty_{x,p}} +  \Vert \delta f_c^{(j)}(t) \Vert_{L^1_{x,p}}
\end{aligned}$$
for every $x \in \mathbb{R}^3$. 
Due to conservation of mass, we have
$$ \Vert \delta f_c^{(j)}(t) \Vert_{L^1_{x,p}} \leq  \Vert f_c^{(j)}(t) \Vert_{L^1_{x,p}} +  \Vert f_c^{(j-1)}(t) \Vert_{L^1_{x,p}} \leq 2  \Vert f^0\Vert_{L^1_{x,p}} \leq  2\Vert \langle x \rangle^{3^+} \langle p \rangle^8 f^0 \Vert_{L^\infty_{x,p}}.$$
Furthermore, expressing the distribution function along characteristics as 
$$f_c^{(k)}(t,x,p) = f^0 \left ( \mathcal{X}_c^{(k)}(0,t,x,p), \mathcal{P}_c^{(k)}(0,t,x,p) \right ) $$
for $k = j-1$ and $k= j$
and using \eqref{eq: bound for Pc}, we find
$$\begin{aligned}
\Vert \langle p \rangle^{3^+} \delta f^{(j)} \Vert_{L^\infty_{x,p}}
& \lesssim \Vert \langle p \rangle^{3^+} \nabla_{x,p} f^0\Vert_{L^\infty_{x,p}}
\left (\delta\mathcal{X}_c^{(j)}(0,t,x,p) + \delta\mathcal{P}_c^{(j)}(0,t,x,p) \right )\\
& \lesssim \Vert \langle x \rangle^{3^+} \langle p \rangle^9 \nabla_{x,p} f^0\Vert_{L^\infty_{x,p}}
\left (\delta\mathcal{X}_c^{(j)}(0,t,x,p) + \delta\mathcal{P}_c^{(j)}(0,t,x,p) \right ).
\end{aligned}$$
Ultimately, combining these estimates yields
$$\begin{aligned}
\delta E_c^{(j+1)}(t, x) & \lesssim 
\Vert \langle x \rangle^{3^+} \langle p \rangle^8 f^0 \Vert_{L^\infty_{x,p}} +
\Vert \langle x \rangle^{3^+} \langle p \rangle^9 \nabla_{x,p} f^0\Vert_{L^\infty_{x,p}}
\left ( \delta \mathcal{X}_c^{(j)}(0,t,x,p) + \delta \mathcal{P}_c^{(j)}(0,t,x,p) \right )\\
& \lesssim \eta \left ( 1+ \delta \mathcal{X}_c^{(j)}(0,t,x,p) + \delta \mathcal{P}_c^{(j)}(0,t,x,p)\right )
\end{aligned}$$
for every $x \in \mathbb{R}^3$.
Estimating on the interval $[0,T]$, taking $\eta$ sufficiently small so that the constant in this inequality is no more than $\eta_0$, and using this in the estimate for the difference of momentum characteristics then yields
$$\begin{aligned}
\delta \mathcal{P}_c^{(j+1)}(s,t,x,p)
& \lesssim \eta_0(1+t-s) \sup_{\substack{0 \leq \tau \leq t \leq T \\x,p \in \mathbb{R}^3 }} \left [ \delta \mathcal{X}_c^{(j)}(\tau,t,x,p) + \delta \mathcal{P}_c^{(j)}(\tau,t,x,p) \right ].
\end{aligned}$$
for every $0\leq s \leq t \leq T$.
Adding this to the estimate for the difference of spatial characteristics, we arrive at 
$$\mcD^{(j+1)} \leq \eta_0(1+T)^2 \mcD^{(j)}$$
where
$$\mcD^{(j)} = \sup_{\substack{0 \leq s \leq t \leq T \\x,p \in \mathbb{R}^3 }} \delta \mathcal{X}_c^{(j)}(s,t,x,p)  +  \sup_{\substack{0 \leq s \leq t \leq T \\x,p \in \mathbb{R}^3 }} \delta \mathcal{P}_c^{(j)}(s,t,x,p).
$$
Taking $\eta_0$ sufficiently small, we find that the sequence of characteristics is contractive, and due to the field estimate above, the sequence of fields is, as well.
Hence, for fixed $\eta_0$ we construct a unique solution on the interval $[0,T]$ for some $T > 0$.

Finally, we extend this solution globally in time using a standard continuous induction argument. Indeed, we denote the solution by $f_c(t,x,p)$ with associated field $E_c(t,x)$ and density $\rho_c(t,x)$
for any $1 \leq c \leq \infty$.
Let 
$$\mu(t) = \sup_{0 \leq s \leq t}\Big\{(1+s)^{\alpha+1}\|E_c(s)\|_{L_x^\infty}+(1+s)^{\alpha+2}\|\nabla_x E_c(s)\|_{L_x^\infty}\Big\}.$$
Passing to the limit in $j$ within \eqref{eq: iterative force field bound assumption}, we find that the field satisfies
$\mu(t) \leq \eta_0$
for $t \in [0,T]$.
Let
$$T_\mathrm{max} = \sup\{t \geq 0: \mu(t) \leq \eta_0 \}.$$
Then, repeating the previous induction argument by using Propositions \ref{prop: density function estimates for perturbed relativistic flows} and \ref{prop: derivative bounds for perturbed relativistic flows} and Lemma \ref{lemma: interpolation inequality}, but applying this to $E_c(t,x)$ and $\rho_c(t,x)$, yields
$$
(1+t)^{\alpha+1} \|E_c(t)\|_{L_x^\infty} + (1+t)^{\alpha+2} \|\partial_{x_k} E_c^{(j+1)}(t)\|_{L_x^\infty} \lesssim 2\eta \leq \frac{1}{2} \eta_0$$
for all $t \in [0,T_\mathrm{max})$
by taking $\eta$ sufficiently small.
Hence, we find $T_\mathrm{max} = \infty$, and this further implies that the solution is global.

In conclusion, under the smallness condition \eqref{eq: initial data assumption2} on initial data $f^0$, we construct a global-in-time solution $f_c(t,x,p)=f^0(\Xi_c(0,t, x,p))$ to the Vlasov equation such that \eqref{eq: small force field assumption} holds, namely
$$\sup_{t\geq0}\Big\{(1+t)^{\alpha+1}\|E_c(t)\|_{L_x^\infty}+(1+t)^{\alpha+2}\|\nabla_x E_c(t)\|_{L_x^\infty}\Big\}\leq\eta_0,$$
and the associated characteristic flow $(\mathcal{X}_c(s,t, x,p), \mathcal{P}_c(s,t, x,p))$ satisfies the equation \eqref{eq: r-ODE} with the force field $E=E_c$. This result includes the non-relativistic case $c=\infty$.
\end{proof}

\subsection{Scattering}\label{sec: scattering}
Next, we prove a result concerning the large-time scattering of solutions along the forward free flow. Similar results of this type, concerning scattering either along the forward free flow or an augmentation of the flow via modified trajectories, have previously been obtained for the Vlasov-Poisson system \cite{IPWW22, Pankavich2, Pankavich1, Pankavich3}, and more recently, the Vlasov-Riesz system \cite{HuangKwon}.
However, motivated by the quantum analogue of the Vlasov equation, we employ the wave operator formulation for the first time to a kinetic model. This simplifies the proof, in particular, when two scattering dynamics are compared. Additionally, the wave operator formulation can also be applied to scattering problems for general kinetic models.

\begin{proof}[Proof of Theorem \ref{T2}]
Let $f_c(t,x,p)$ be the small-data global solution to the Vlasov equation with initial data $f^0$, constructed in the previous subsection. Then, we expect
\begin{equation}\label{eq: original convergence problem}
\begin{aligned}
f_c\Big(t,x+tv_c(p),p\Big)&=f^0\Big(\big(\Phi_c(t)^{-1}\circ\Phi_c^{\textup{free}}(t)\big)(x,p)\Big)=f^0\Big(\mathcal{W}_c(t)^{-1}(x,p)\Big)\\
&\to f^0\Big((\mathcal{W}_c^+)^{-1}(x,p)\Big)
\end{aligned}
\end{equation}
as $t\to\infty$.  Hence, it is natural to define the limiting profile by
$$f_c^+:= f^0\circ (\mathcal{W}_c^+)^{-1}.$$
Our goal is then to show that 
$$\lim_{t\to\infty}\Big\|f_c\Big(t,x+tv_c(p),p\Big)-f_c^+(x,p)\Big\|_{L_{x,p}^1}=0.$$
For this, we again denote
\begin{equation}\label{eq: gc definition}
g_c(t,x,p):=f_c\Big(t,x+tv_c(p),p\Big)=f^0\Big(\mathcal{W}_c(t)^{-1}(x,p)\Big)
\end{equation}
so that 
$$f_c^+= f^0\circ (\mathcal{W}_c^+)^{-1}=g_c(t)\circ \mathcal{W}_c(t)\circ (\mathcal{W}_c^+)^{-1}.$$
Then, by the volume preserving property of the wave operator $\mathcal{W}_c^+$, the norm of the difference can be written as 
$$\begin{aligned}
\Big\|f_c\Big(t,x+tv_c(p),p\Big)-f_c^+(x,p)\Big\|_{L_{x,p}^1}&=\Big\|g_c(t,x,p)-g_c\Big(t,\mathcal{W}_c(t)\circ (\mathcal{W}_c^+)^{-1}(x,p)\Big)\Big\|_{L_{x,p}^1}\\
&=\Big\|g_c\Big(t,\mathcal{W}_c^+(x,p)\Big)-g_c\Big(t,\mathcal{W}_c(t)(x,p)\Big)\Big\|_{L_{x,p}^1}.
\end{aligned}$$

\begin{remark}
By writing the solution in this way, we can avoid dealing with inverse maps of the wave operator $(\mathcal{W}_c^+)^{-1}$ and $\mathcal{W}_c(t)^{-1}$ contrary to \eqref{eq: original convergence problem}. 
\end{remark}

Now, for $0\leq\theta\leq 1$, we define the interpolated map $\mathcal{W}_c^\theta(t)$ by 
\begin{equation}
\label{Wtheta}
\mathcal{W}_c^\theta(t):=\theta\mathcal{W}_c^++(1-\theta)\mathcal{W}_c(t).
\end{equation}
This yields
$$\begin{aligned}
&g_c\Big(t,\mathcal{W}_c^+(x,p)\Big)-g_c\Big(t,\mathcal{W}_c(t)(x,p)\Big)\\
&=\int_0^1\frac{d}{d\theta} g_c\Big(t,\mathcal{W}_c^\theta(t)(x,p)\Big)d\theta
=\int_0^1\nabla_{(x,p)}g_c\Big(t,\mathcal{W}_c^\theta(t)(x,p)\Big)\frac{d}{d\theta}\mathcal{W}_c^\theta(t)(x,p) d\theta\\
&=\bigg\{\int_0^1\nabla_{(x,p)}g_c\Big(t,\mathcal{W}_c^\theta(t)(x,p)\Big) d\theta\bigg\} \big(\mathcal{W}_c^+-\mathcal{W}_c(t)\big)(x,p).
\end{aligned}$$
Hence, it follows from the convergence of the wave operator \eqref{eq: rate of convergence to the wave operator} that 
$$\Big\|f_c\Big(t,x+tv_c(p),p\Big)-f_c^+(x,p)\Big\|_{L_{x,p}^1}\lesssim\frac{1}{(1+t)^\alpha}\int_0^1\Big\|\nabla_{(x,p)}g_c\Big(t,\mathcal{W}_c^\theta(t)(x,p)\Big)\Big\|_{L_{x,p}^1} d\theta.$$
Due to Lemma \ref{lemma:  the wave operator, almost identity} we obtain for the interpolated map
\begin{equation}\label{eq: derivative of W theta estimate}
|\nabla_{(x,p)}\mathcal{W}_c^\theta(t)-\mathbb{I}_6|\leq\theta|\nabla_{(x,p)}\mathcal{W}_c^+-\mathbb{I}_6|+(1-\theta)|\nabla_{(x,p)}\mathcal{W}_c(t)-\mathbb{I}_6|\lesssim\eta_0.
\end{equation}
Thus, changing variables by $(y,w)=\mathcal{W}_c^\theta(t)(x,p)$ with $|\textup{det} \left (\frac{\partial (y,w)}{\partial(x,p)} \right )|\lesssim 1$, we obtain 
$$\Big\|\nabla_{(x,p)}g_c\Big(t,\mathcal{W}_c^\theta(t)(x,p)\Big)\Big\|_{L_{x,p}^1}\lesssim\|\nabla_{(x,p)}g_c(t,x,p)\|_{L_{x,p}^1}.$$
Recalling the definition of $g_c$ in \eqref{eq: gc definition}, we have 
$$\nabla_{(x,p)}g_c(t,x,p)=\nabla_{(x,p)}f^0\Big(\mathcal{W}_c(t)^{-1}(x,p)\Big)\big(\nabla_{(x,p)}\mathcal{W}_c(t)^{-1}\Big)(x,p).$$
Therefore, by \eqref{eq: approx identity, inverse finite-time wave operator} and the volume preserving property of $\mathcal{W}_c(t)$, we find 
$$\big\|\nabla_{(x,p)}g_c\big(t,x,p)\big\|_{L_{x,p}^1}\lesssim\Big\|\nabla_{(x,p)}f^0\Big(\mathcal{W}_c(t)^{-1}(x,p)\Big)\Big\|_{L_{x,p}^1}=\|\nabla_{(x,p)}f^0\|_{L_{x,p}^1}.$$
Collecting these estimates, we finally conclude
$$\Big\|f_c\Big(t,x+tv_c(p),p\Big)-f_c^+(x,p)\Big\|_{L_{x,p}^1}\lesssim\frac{1}{(1+t)^\alpha}\|\nabla_{(x,p)}f^0\|_{L_{x,p}^1}.$$
\end{proof}

\begin{remark}
The same proof works in the non-relativistic case $c=\infty$ by merely replacing $v_c(p)$ with $p$ throughout. 
\end{remark}

\section{Non-relativistic limit for the Vlasov equation and scattering states}\label{sec: Non-relativistic limit for the Vlasov equation and scattering states}

The main goals of this section are to prove that solutions of the relativistic system translated along the forward free flow converge to their non-relativistic analogues as $c \to \infty$ and that the associated scattering states converge in the same limit.

\subsection{Non-relativistic limit for the Vlasov equation}
First, we prove the following result, which guarantees the convergence of solutions to the relativistic system in the limit as $c \to \infty$ for large times.

\begin{proposition}[Non-relativistic limit for the Vlasov equation]\label{prop: nr limit for Vlasov}
Under the assumptions of Theorems \ref{T1} and \ref{T2}, let
$$f_c(t,x,p)=f^0\Big(\Phi_c(t)^{-1}(x,p)\Big)$$
for every $1 \leq c \leq \infty$
be the global solution to the relativistic (or non-relativistic) Vlasov equation with initial data,  constructed in Theorem \ref{T1}. Then, for $t\geq 1$, we have
\begin{equation}\label{eq: nr limit for Vlasov}
\|g_c(t,x,p)-g_\infty(t,x,p)\|_{L_{x,p}^1(\mathbb{R}^6)}\lesssim \frac{1}{c^2} \|\langle p \rangle^3\nabla_{(x,p)}f^0\|_{L_{x,p}^1(\mathbb{R}^6)},
\end{equation}
where $g_c(t,x,p)=f_c(t,x+tv_c(p),p)$. Moreover, the wave operator and force field obey the bounds
\begin{equation}\label{eq: nr limit for wave operators}
\sup_{t\geq0}\bigg\|\frac{1}{\langle p\rangle^3}\Big(\mathcal{W}_c(t)(x,p)-\mathcal{W}_\infty(t)(x,p)\Big)\bigg\|_{C_{x,p}(\mathbb{R}^6)}\lesssim \frac{\eta_0}{c^2}
\end{equation}
and
\begin{equation}\label{eq: nr limit for force field}
\|(E_c-E_\infty)(t)\|_{L_x^\infty(\mathbb{R}^3)}\lesssim\frac{\eta_0}{c^2\langle t\rangle^{\alpha+1}},
\end{equation}
respectively.
\end{proposition}

For the proof, we estimate the difference between the respective wave operators (Lemma \ref{lemma: primitive bound for the difference between wave operators}) and between the force fields (Lemma \ref{lemma: difference between force fields}). Then, combining them, we obtain the desired convergence estimates \eqref{eq: nr limit for wave operators} and \eqref{eq: nr limit for force field}.

\begin{lemma}\label{lemma: primitive bound for the difference between wave operators}
Under the assumptions of Proposition \ref{prop: nr limit for Vlasov}, we have
$$\Big\|\frac{1}{\langle p\rangle^3}\Big(\mathcal{W}_{c}(t)(x,p)-\mathcal{W}_{\infty}(t)(x,p)\Big)\Big\|_{C_{x,p}(\mathbb{R}^6)}\lesssim \frac{\eta_0}{c^2}+\big\|(1+\tau)(E_c-E_\infty)(\tau)\big\|_{L_{\tau}^1([0,t]; L_x^\infty)}.$$
\end{lemma}

\begin{proof}
Throughout the proof, we fix $x$ and $p$ and omit them in $\mathcal{W}_c(t)$ and $\mathcal{W}_\infty(t)$ for brevity. Also, we denote $\mathcal{X}_c(t)=\mathcal{X}_c(t,0,x,p)$ and $\mathcal{P}_c(t)=\mathcal{P}_c(t,0,x,p)$. Using \eqref{eq: explicit representation of the finite-time wave operator-2}, one can write the difference between the wave operators as
$$\begin{aligned}
\mathcal{W}_c(t)-\mathcal{W}_\infty(t)
&=-\int_0^t\tau\Big[\mathbb{A}_c\big(\mathcal{P}_c(\tau)\big)-\mathbb{I}_3, 0\Big]^T E_c\big(\tau,\mathcal{X}_\infty(\tau)\big)d\tau\\
&\quad-\int_0^t \Big[\tau\mathbb{I}_3, -\mathbb{I}_3\Big]^T\Big\{E_c\big(\tau,\mathcal{X}_c(\tau)\big)-E_\infty\big(\tau,\mathcal{X}_\infty(\tau)\big)\Big\}dt\tau\\
&=:\textup{(I)}+\textup{(II)},
\end{aligned}$$
where $[\mathbb{A}_c(\mathcal{P}_c(\tau))-\mathbb{I}_3, 0]$ and $[\tau\mathbb{I}_3, -\mathbb{I}_3]$ are $3\times 6$ matrices. To estimate $\textup{(I)}$, we use \eqref{eq: Ac asymptotic'} and \eqref{eq: bound for Pc} to obtain 
$$\big\|\mathbb{A}_c\big(\mathcal{P}_c(\tau)\big)-\mathbb{I}_3\big\|\lesssim\frac{|p|^2}{c^2}$$
for any $0 \leq \tau \leq t$. For $\textup{(II)}$, we separate the difference of the fields into
$$\begin{aligned}
&E_c\big(\tau,\mathcal{X}_c(\tau)\big)-E_\infty\big(\tau,\mathcal{X}_\infty(\tau)\big)\\
&=\big(E_c-E_\infty\big)(\tau,\mathcal{X}_c(\tau))+\Big\{E_\infty\big(\tau,\mathcal{X}_c(\tau)\big)-E_\infty\big(\tau,\mathcal{X}_\infty(\tau)\big)\Big\}.
\end{aligned}$$
Assembling these estimates yields
$$\begin{aligned}
\big|\mathcal{W}_{c}(t)-\mathcal{W}_{\infty}(t)\big|&\lesssim \frac{|p|^2}{c^2}\|\tau E_\infty(\tau)\|_{L_{\tau}^1([0,t]; L_x^\infty)}+\big\|(1+\tau)\big(E_c-E_\infty\big)(\tau)\big\|_{L_{\tau}^1([0,t]; L_x^\infty)}\\
&\quad+\int_0^t(1+\tau)\|\nabla_x E_\infty(\tau)\|_{L_x^\infty}|\mathcal{X}_c(\tau)-\mathcal{X}_\infty(\tau)|d\tau.
\end{aligned}$$
Now, for the last term in the upper bound, recalling $\mathcal{X}_c(\tau)=\mathcal{W}_{c;1}(\tau)+\tau v_c(\mathcal{W}_{c;2}(\tau))$ from Definition \ref{def: classical finite-time wave operator}, we note that 
$$\begin{aligned}
\mathcal{X}_c(\tau)-\mathcal{X}_\infty(\tau)&=\big(\mathcal{W}_{c;1}(\tau)-\mathcal{W}_{\infty;1}(\tau)\big)+\tau \Big\{v_c\big(\mathcal{W}_{c;2}(\tau)\big)-\mathcal{W}_{c;2}(\tau)\Big\}\\
&\quad+\tau \big(\mathcal{W}_{c;2}(\tau)-\mathcal{W}_{\infty;2}(\tau)\big),
\end{aligned}$$
and thus, by \eqref{eq: vc asymptotic}, \eqref{eq: bound for Pc}, and \eqref{eq: approx identity, finite-time wave operator}, we find
$$|\mathcal{X}_c(\tau)-\mathcal{X}_\infty(\tau)|\lesssim(1+\tau)|\mathcal{W}_{c}(\tau)-\mathcal{W}_{\infty}(\tau)|+\frac{\tau\langle p\rangle^3}{c^2}.$$
Therefore, it follows that 
$$\begin{aligned}
\frac{1}{\langle p\rangle^3}|\mathcal{W}_{c}(t)-\mathcal{W}_{\infty}(t)|&\lesssim \frac{1}{c^2}\|\tau E_\infty(\tau)\|_{L_{\tau}^1([0,t]; L_x^\infty)}+\big\|(1+\tau)(E_c-E_\infty)(\tau)\big\|_{L_{\tau}^1([0,t]; L_x^\infty)}\\
&\quad+\frac{1}{c^2}\|(1+\tau)^2 \nabla_x E_\infty(\tau)\|_{L_{\tau}^1([0,t]; L_x^\infty)}\\
&\quad+\int_0^t(1+\tau)^2\|\nabla_x E_\infty(\tau)\|_{L_x^\infty}\frac{1}{\langle p\rangle^3}|(\mathcal{W}_c(\tau)-\mathcal{W}_\infty(\tau))|d\tau.
\end{aligned}$$
Finally, applying Gr\"onwall's inequality to $\frac{1}{\langle p\rangle^3}|\mathcal{W}_{c}(t)-\mathcal{W}_{\infty}(t)|$ with the decay bound \eqref{eq: uniform force field decay bounds}, we obtain the desired bound.
\end{proof}

Conversely, we prove a bound for the difference between the force fields using the difference of the wave operators.

\begin{lemma}\label{lemma: difference between force fields}
Under the assumptions of Proposition \ref{prop: nr limit for Vlasov}, we have
$$\begin{aligned}
\|(E_c-E_\infty)(t)\|_{L_x^\infty(\mathbb{R}^3)}&\lesssim\frac{\eta_0}{(1+t)^{\alpha+1}}\bigg\{\frac{1}{c^2}+\Big\|\frac{1}{\langle p\rangle^3}\Big(\mathcal{W}_c(t)(x,p)-\mathcal{W}_\infty(t)(x,p)\Big)\Big\|_{C_{x,p}(\mathbb{R}^6)}\bigg\}.
\end{aligned}$$
\end{lemma}

For the proof, it is convenient to first introduce the interpolated wave operator given by 
\begin{equation}\label{eq: interpolated wave operator}
\widetilde{\mathcal{W}}^\theta(t)=\big(\widetilde{\mathcal{W}}_1^\theta(t),\widetilde{\mathcal{W}}_2^\theta(t)\big):=\theta\mathcal{W}_c(t)+(1-\theta)\mathcal{W}_\infty(t):\mathbb{R}^3\times\mathbb{R}^3\to\mathbb{R}^3\times\mathbb{R}^3
\end{equation}
for $0\leq\theta\leq1$. It is important to note that this interpolated operator is different from the one used in the proof of scattering \eqref{Wtheta}. That being said, it will be used similarly herein.

\begin{lemma}[Interpolated wave operator]\label{lemma: interpolated wave operator}
Under the assumptions of Proposition \ref{prop: nr limit for Vlasov}, we have
\begin{equation}\label{eq: interpolated wave operator, almost identity}
\sup_{t\geq0}\big\|\widetilde{\mathcal{W}}^\theta(t)(x,p)-(x,p)\big\|_{C_{x,p}(\mathbb{R}^6)}+\sup_{t\geq0}\big\|\nabla_{(x,p)}\widetilde{\mathcal{W}}^\theta(t)(x,p)-\mathbb{I}_6\big\|_{C_{x,p}(\mathbb{R}^6)}\lesssim\eta_0, 
\end{equation}
\end{lemma}

\begin{proof}
Because 
$\widetilde{\mathcal{W}}^\theta(t)(x,p)-(x,p)=\theta(\mathcal{W}_c(t)(x,p)-(x,p))+(1-\theta)(\mathcal{W}_\infty(t)(x,p)-(x,p))$, the lemma follows from \eqref{eq: approx identity, finite-time wave operator}. 
\end{proof}

\begin{proof}[Proof of Lemma \ref{lemma: difference between force fields}]
\textbf{Step 1 - Density Estimate}

\noindent Recalling that $f_c(t,x,p)=g_c(t,x-tv_c(p),p)$ and 
$g_c(t,x,p)=f^0(\mathcal{W}_c(t)^{-1}(x,p))$ for $1\leq c\leq\infty$, we write the difference between the distribution functions as 
$$\begin{aligned}
(f_c-f_\infty)(t,x,p)&=g_c\Big(t,x-tv_c(p),p\Big)-g_\infty\Big(t,x-tp,p\Big)\\
&=(g_c-g_\infty)\Big(t,x-tv_c(p),p\Big)+\Big\{g_\infty\Big(t,x-tv_c(p),p\Big)-g_\infty\Big(t,x-tp,p\Big)\Big\}\\
&=\textup{(A)}+\textup{(B)}.
\end{aligned}$$
Then, introducing $v_c^\theta(p)=\theta v_c(p)+(1-\theta)p$, we have 
$$\textup{(B)}=\int_0^1 \frac{d}{d\theta}g_\infty\Big(t,x-tv_c^\theta(p),p\Big)d\theta=t\int_0^1 \nabla_xg_\infty\Big(t,x-tv_c^\theta(p),p\Big)\cdot\big(p-v_c(p)\big)d\theta.$$
Notice that
%\footnote{Precisely, $$\begin{aligned}
%\begin{bmatrix}
%\nabla_{p_1}\{g_\infty(\cdots)\}\\
%\nabla_{p_2}\{g_\infty(\cdots)\}\\
%\nabla_{p_3}\{g_\infty(\cdots)\}
%\end{bmatrix}&=-t\begin{bmatrix}
%\nabla_{p_1}v_{\theta;1}(p)& \nabla_{p_1}v_{\theta;2}(p)& \nabla_{p_1}v_{\theta;3}(p)\\
%\nabla_{p_2}v_{\theta;1}(p)& \nabla_{p_2}v_{\theta;2}(p)& \nabla_{p_2}v_{\theta;3}(p)\\
%\nabla_{p_3}v_{\theta;1}(p)& \nabla_{p_3}v_{\theta;2}(p)& \nabla_{p_3}v_{\theta;3}(p) \end{bmatrix}\begin{bmatrix}
%(\nabla_{x_1}g_\infty)(\cdots)\\
%(\nabla_{x_2}g_\infty)(\cdots)\\
%(\nabla_{x_3}g_\infty)(\cdots)
%\end{bmatrix}+\begin{bmatrix}
%(\nabla_{p_1}g_\infty)(\cdots)\\
%(\nabla_{p_2}g_\infty)(\cdots)\\
%(\nabla_{p_3}g_\infty)(\cdots)
%\end{bmatrix},
%\end{aligned}$$
%where $(\cdots)=(t,x-tv_c^\theta(p),p)$.}
$$\begin{aligned}
\nabla_p\Big(g_\infty\Big(t,x-tv_c^\theta(p),p\Big)\Big)&=-t\Big[\nabla_p v_c^\theta(p)\Big]\nabla_xg_\infty\Big(t,x-tv_c^\theta(p),p\Big)\\
&\quad+\nabla_pg_\infty\Big(t,x-tv_c^\theta(p),p\Big),
\end{aligned}$$
and $\nabla v_c^\theta(p)=\theta \nabla v_c(p)+(1-\theta)\mathbb{I}_3$ is invertible, because Lemma \ref{lemma: Ac eigenvalues} ensures that the symmetric matrix $\nabla v_c(p)=\mathbb{A}_c(p)$ is positive definite. Thus, we have
$$\begin{aligned}
&\nabla_xg_\infty\Big(t,x-tv_c^\theta(p),p\Big)\cdot\big(p-v_c(p)\big)\\
&=\bigg\langle\frac{1}{t}\Big[\nabla_p v_c^\theta(p)\Big]^{-1}\bigg[\nabla_p g_\infty\Big(t,x-tv_c^\theta(p),p\Big)-\nabla_p\Big(g_\infty\Big(t,x-tv_c^\theta(p),p\Big)\Big)\bigg],\big(p-v_c(p)\big)\bigg\rangle_{\mathbb{R}^3}\\
&=\frac{1}{t}\bigg\langle\bigg[\nabla_pg_\infty\Big(t,x-tv_c^\theta(p),p\Big)-\nabla_p\Big(g_\infty\Big(t,x-tv_c^\theta(p),p\Big)\Big)\bigg],\Big[\nabla_p v_c^\theta(p)\Big]^{-1}\big(p-v_c(p)\big)\bigg\rangle_{\mathbb{R}^3},
\end{aligned}$$
where both $a\cdot b$ and $\langle a, b\rangle_{\mathbb{R}^3}$ are used to denote the inner product in $\mathbb{R}^3$. We insert this within $\textup{(B)}$ and integrate $(f_c-f_\infty)(t,x,p)$ in $p$ over $\mathbb{R}^3$ while maintaining the previous expression for $(\textup{A})$. Then, upon integrating by parts in the last term, the difference between the density functions can be decomposed as 
$$\begin{aligned}
(\rho_{f_c}-\rho_{f_\infty})(t,x)&=\int_{\mathbb{R}^3}(g_c-g_\infty)\Big(t,x-tv_c(p),p\Big)dp\\
&\quad+\int_0^1\int_{\mathbb{R}^3}\nabla_pg_\infty\Big(t,x-tv_c^\theta(p),p\Big)\cdot w_\theta(p)dpd\theta\\
&\quad+\int_0^1\int_{\mathbb{R}^3}g_\infty\Big(t,x-tv_c^\theta(p),p\Big)\nabla_p\cdot w_\theta(p)dpd\theta\\
&=\textup{(I)}+\textup{(II)}+\textup{(III)},
\end{aligned}$$
where 
$$w_\theta(p):=\Big[\nabla_p v_c^\theta(p)\Big]^{-1}\big(p-v_c(p)\big)$$
For $\textup{(II)}$ and $\textup{(III)}$, we note that  
$$w_\theta(p)=\Big[\nabla_p v_c^\theta(p)\Big]^{-1}\bigg(1-\frac{1}{\gamma_c(p)}\bigg)p=\frac{1-\frac{1}{\gamma_c(p)}}{\frac{\theta}{\gamma_c(p)^3}+1-\theta}p,$$
because by the identity 
$$\mathbb{A}_c(p)p=\frac{1}{\gamma_c(p)}p-\frac{\frac{|p|^2}{c^2}}{\gamma_c(p)^3}p=\frac{1}{\gamma_c(p)^3}p,$$
which follows from \eqref{Adef}, we have
$$[\nabla_p v_c^\theta(p)]p=\theta \mathbb{A}_c(p)p+(1-\theta)p=\bigg(\frac{\theta}{\gamma_c(p)^3}+(1-\theta)\bigg)p$$
and $\nabla_p v_c^\theta(p)$ is an invertible matrix. Hence, we have 
$$|w_\theta(p)|=\frac{1-\frac{1}{\gamma_c(p)}}{\frac{\theta}{\gamma_c(p)^3}+1-\theta}|p|\leq\frac{\frac{\gamma_c(p)^2-1}{\gamma_c(p)(\gamma_c(p)+1)}}{\frac{1}{\gamma_c(p)^3}}|p|\leq\frac{\gamma_c(p)|p|^3}{c^2}$$
and its divergence 
$$\nabla_p\cdot w_\theta(p)=\frac{\frac{\nabla\gamma_c(p)}{\gamma_c(p)^2}}{\frac{\theta}{\gamma_c(p)^3}+1-\theta}\cdot p-\frac{1-\frac{1}{\gamma_c(p)}}{(\frac{\theta}{\gamma_c(p)^3}+1-\theta)^2}\frac{3\theta(\nabla\gamma_c(p))}{\gamma_c(p)^4}\cdot p+\frac{3(1-\frac{1}{\gamma_c(p)})}{\frac{\theta}{\gamma_c(p)^3}+1-\theta}$$
satisfies
$$\big|\nabla_p\cdot w_\theta(p)\big|\lesssim\frac{\gamma_c(p)|p|^2}{c^2}.$$
Thus, applying the bounds for $|w_\theta(p)|$ and $|\nabla_p\cdot w_\theta(p)|$ to $\textup{(II)}$ and $\textup{(III)}$ respectively, it follows that 
$$\begin{aligned}
\big|(\rho_{f_c}-\rho_{f_\infty})(t,x)\big|&\lesssim\int_{\mathbb{R}^3}(|g_c-g_\infty|)\Big(t,x-tv_c(p),p\Big)dp\\
&\quad+\frac{1}{c^2}\int_0^1\int_{\mathbb{R}^3}\Big(\gamma_c(p)|p|^3|\nabla_pg_\infty|\Big)\Big(t,x-tv_c^\theta(p),p\Big)dpd\theta\\
&\quad+\frac{1}{c^2}\int_0^1\int_{\mathbb{R}^3}\Big(\gamma_c(p)|p|^2|g_\infty|\Big)\Big(t,x-tv_c^\theta(p),p\Big)dpd\theta.
\end{aligned}$$
As a consequence, by Lemma \ref{lemma: dispersive bounds for the free flow associated with interpolated v}, we obtain 
$$\begin{aligned}
\big\|(\rho_{f_c}-\rho_{f_\infty})(t)\big\|_{L_x^1}&\lesssim\|(g_c-g_\infty)(t)\|_{L_{x,p}^1}+\frac{1}{c^2}\big\|\gamma_c(p)|p|^3\nabla_pg_\infty(t)\big\|_{L_{x,p}^1}\\
&\quad+\frac{1}{c^2}\big\|\gamma_c(p)|p|^2g_\infty(t)\big\|_{L_{x,p}^1}
\end{aligned}$$
and
$$\begin{aligned}
\big\|(\rho_{f_c}-\rho_{f_\infty})(t)\big\|_{L_x^\infty}&\lesssim\frac{1}{(1+t)^3}\Big\|\langle x\rangle^{3^+}\langle p\rangle^{5}(g_c-g_\infty)(t)\Big\|_{L^\infty_{x,p}}\\
&\quad+\frac{1}{(1+t)^3c^2}\Big\|\langle x\rangle^{3^+}\langle p\rangle^{5}\gamma_c(p)|p|^3\nabla_pg_\infty(t)\Big\|_{L^\infty_{x,p}}\\
&\quad+\frac{1}{(1+t)^3c^2}\Big\|\langle x\rangle^{3^+}\langle p\rangle^{5}\gamma_c(p)|p|^2g_\infty(t)\Big\|_{L^\infty_{x,p}}.
\end{aligned}$$
Using Lemma \ref{lemma:  the wave operator, almost identity} with the change of variables $(x,p)=\mathcal{W}_\infty(t)(\tilde{x},\tilde{p})$ and Lemma \ref{lemma: perturbed momentum}, we find
$$\begin{aligned}
\big\|\gamma_c(p)|p|^3\nabla_pg_\infty(t)\big\|_{L_{x,p}^1}&=\Big\|\gamma_c(p)|p|^3(\nabla_{(x,p)}f^0)\Big(\mathcal{W}_\infty(t)^{-1}(x,p)\Big)\cdot\nabla_p\big(\mathcal{W}_\infty(t)^{-1}\big)(x,p)\Big\|_{L_{x,p}^1}\\
&\lesssim \Big\|\gamma_c\big(p(\tilde{x},\tilde{p})\big)|p(\tilde{x},\tilde{p})|^3(\nabla_{(x,p)}f^0)(\tilde{x},\tilde{p})\Big\|_{L_{\tilde{x},\tilde{p}}^1}\\
&\lesssim \big\|\gamma_c(p)|p|^3\nabla_{(x,p)}f^0\big\|_{L_{x,p}^1}.
\end{aligned}$$
In the same manner, one can show
$$\begin{aligned}
\big\|\gamma_c(p)|p|^2g_\infty(t)\big\|_{L_{x,p}^1}&\lesssim \big\|\gamma_c(p)|p|^2f^0\big\|_{L_{x,p}^1}\leq \Big\|\langle x\rangle^{3^+}\langle p\rangle^{9}\nabla_{(x,p)}f^0\Big\|_{L^\infty_{x,p}},\\
\Big\|\langle x\rangle^{3^+}\langle p\rangle^{5}\gamma_c(p)|p|^3\nabla_pg_\infty(t)\Big\|_{L^\infty_{x,p}}&\lesssim \Big\|\langle x\rangle^{3^+}\langle p\rangle^{5}\gamma_c(p)|p|^3\nabla_{(x,p)}f^0\Big\|_{L^\infty_{x,p}}
\leq \Big\|\langle x\rangle^{3^+}\langle p\rangle^{9}\nabla_{(x,p)}f^0\Big\|_{L^\infty_{x,p}},\\
\Big\|\langle x\rangle^{3^+}\langle p\rangle^{5}\gamma_c(p)|p|^2g_\infty(t)\Big\|_{L^\infty_{x,p}}&\lesssim \Big\|\langle x\rangle^{3^+}\langle p\rangle^{5}\gamma_c(p)|p|^2f^0\Big\|_{L^\infty_{x,p}}
\leq \Big\|\langle x\rangle^{3^+}\langle p\rangle^{8}f^0\Big\|_{L^\infty_{x,p}}.
\end{aligned}$$
Then, by the smallness assumption on the initial data \eqref{eq: initial data assumption}, it follows that  
\begin{equation}\label{eq: L1 Linfty rho difference bounds}
\begin{aligned}
\big\|(\rho_{f_c}-\rho_{f_\infty})(t)\big\|_{L_x^1}&\lesssim\|g_c-g_\infty\|_{L_{x,p}^1}+\frac{\eta_0}{c^2},\\
\big\|(\rho_{f_c}-\rho_{f_\infty})(t)\big\|_{L_x^\infty}&\lesssim\frac{1}{(1+t)^3}\bigg(\Big\|\Big(\langle x\rangle^{3^+}\langle p\rangle^5(g_c-g_\infty)(t)\Big\|_{L^\infty_{x,p}}+\frac{\eta_0}{c^2}\bigg).
\end{aligned}
\end{equation}
\textbf{Step 2 - Difference of Distributions} 

\noindent Next, we estimate $\|g_c-g_\infty\|_{L_{x,p}^1}$ and
$\|\langle x\rangle^{3^+}\langle p\rangle^{5}(g_c-g_\infty)(t)\|_{L^\infty_{x,p}}$
in the upper bounds \eqref{eq: L1 Linfty rho difference bounds} remaining from the previous step. To do so, we first note that 
$$g_c(t,x,p)=g_\infty\Big(t,\big(\mathcal{W}_\infty(t)\circ\mathcal{W}_c(t)^{-1}\big)(x,p)\Big).$$
Hence, using the volume-preserving property of $\mathcal{W}_c(t)$ with the interpolated wave operator \eqref{eq: interpolated wave operator}, it follows that 
$$\begin{aligned}
\|(g_c-g_\infty)
(t)\|_{L_{x,p}^1}&=\bigg\|\int_0^1 \frac{d}{d\theta}\Big[g_\infty\Big (t, \widetilde{\mathcal{W}}^\theta(t)(x,p)\Big)\Big] d\theta\bigg\|_{L_{x,p}^1}\\
&\leq\int_0^1\Big\|\nabla_{(x,p)}g_\infty\Big(t,\widetilde{\mathcal{W}}^\theta(t)(x,p)\Big)\cdot\Big(\mathcal{W}_c(t)-\mathcal{W}_\infty(t)\Big)(x,p)\Big\|_{L_{x,p}^1}d\theta\\
&\leq\bigg\|\frac{\mathcal{W}_c(t)-\mathcal{W}_\infty(t)}{\langle p\rangle^3}\bigg\|_{C_{x,p}}\int_0^1\Big\|\langle p\rangle^3\nabla_{(x,p)}g_\infty\Big(t,\widetilde{\mathcal{W}}^\theta(t)(x,p)\Big)\Big\|_{L_{x,p}^1}d\theta.
\end{aligned}$$
For the second factor in the upper bound, we use Lemma \ref{lemma: interpolated wave operator} and change variables via $$(\tilde{x},\tilde{p})=\widetilde{\mathcal{W}}^\theta(t)(x,p)=\Big(\widetilde{\mathcal{W}}_1^\theta(t)(x,p), \widetilde{\mathcal{W}}_2^\theta(t)(x,p)\Big): \mathbb{R}^3\times\mathbb{R}^3\to \mathbb{R}^3\times\mathbb{R}^3,$$ 
to find 
$$\Big\|\langle p\rangle^3\nabla_{(x,p)}g_\infty\Big(t,\widetilde{\mathcal{W}}^\theta(t)(x,p)\Big)\Big\|_{L_{x,p}^1}\lesssim\big\|\langle \tilde{p}\rangle^3\nabla_{(x,p)}g_\infty(t, \tilde{x}, \tilde{p})\big\|_{L_{\tilde{x},\tilde{p}}^1}.$$
Subsequently, applying the chain rule to $g_\infty(t,\tilde{x},\tilde{p})=f^0(\mathcal{W}_\infty(t)^{-1}(\tilde{x},\tilde{p}))$ and using \eqref{eq: approx identity, inverse finite-time wave operator}, we obtain
$$\begin{aligned}
&\Big\|\langle p\rangle^3\nabla_{(x,p)}g_\infty\Big(t,\widetilde{\mathcal{W}}^\theta(t)(x,p)\Big)\Big\|_{L_{x,p}^1}\\
&\lesssim\Big\|\langle \tilde{p}\rangle^3\nabla_{(x,p)}f^0\Big(\mathcal{W}_\infty(t)^{-1}(\tilde{x},\tilde{p})\Big)\Big\|_{L_{\tilde{x},\tilde{p}}^1}\Big\|\nabla_{(\tilde{x},\tilde{p})}\Big(\mathcal{W}_\infty(t)^{-1}(\tilde{x},\tilde{p})\Big)\Big\|_{C_{\tilde{x},\tilde{p}}}\\
&\lesssim \big\|\langle p\rangle^3\nabla_{(x,p)}f^0 \big\|_{L_{x,p}^1}
\leq \Big\|\langle x\rangle^{3^+}\langle p \rangle^9\nabla_{(x,p)}f^0\Big\|_{L_{x,p}^\infty} \leq\eta_0.
\end{aligned}$$
Therefore, we find
\begin{equation}\label{eq: difference between force fields-1}
\|(g_c-g_\infty)(t)\|_{L_{x,p}^1}\lesssim \eta_0\Big\|\frac{1}{\langle p\rangle^3}\Big(\mathcal{W}_c(t)(x,p)-\mathcal{W}_\infty(t)(x,p)\Big)\Big\|_{C_{x,p}}.
\end{equation}
Using the same tools, it follows that 
\begin{equation}\label{eq: difference between force fields-2}
\Big\|\langle x\rangle^{3^+}\langle p\rangle^5(g_c-g_\infty)(t)\Big\|_{L^\infty_{x,p}}\lesssim \eta_0\Big\|\frac{1}{\langle p\rangle^3}\Big(\mathcal{W}_c(t)(x,p)-\mathcal{W}_\infty(t)(x,p)\Big)\Big\|_{C_{x,p}}
\end{equation}
as
$\Big\|\langle x\rangle^{3^+}\langle p \rangle^8\nabla_{(x,p)}f^0\Big\|_{L_{x,p}^\infty}\leq \eta_0.$
Finally, inserting \eqref{eq: difference between force fields-1} and \eqref{eq: difference between force fields-2} within \eqref{eq: L1 Linfty rho difference bounds}, we arrive at the estimates
$$
\begin{aligned}
\big\|(\rho_{f_c}-\rho_{f_\infty})(t)\big\|_{L_x^1}&\lesssim
\eta_0 \bigg(\frac{1}{c^2} + \Big\|\frac{1}{\langle p\rangle^3}\Big(\mathcal{W}_c(t)(x,p)-\mathcal{W}_\infty(t)(x,p)\Big)\Big\|_{C_{x,p}} \bigg),\\
\big\|(\rho_{f_c}-\rho_{f_\infty})(t)\big\|_{L_x^\infty}&\lesssim
\frac{\eta_0}{(1+t)^3}\bigg(\frac{1}{c^2} + \Big\|\frac{1}{\langle p\rangle^3}\Big(\mathcal{W}_c(t)(x,p)-\mathcal{W}_\infty(t)(x,p)\Big)\Big\|_{C_{x,p}} \bigg).
\end{aligned}
$$
Using the interpolation inequality in Lemma \ref{lemma: interpolation inequality} for the difference $(\rho_{f_c}-\rho_{f_\infty})(t)$, the proof is complete.
\end{proof}

\begin{proof}[Proof of Proposition \ref{prop: nr limit for Vlasov}]
By Lemmas \ref{lemma: primitive bound for the difference between wave operators} and \ref{lemma: difference between force fields}, we obtain
$$\begin{aligned}
\sup_{0\leq \tau\leq t}\bigg\|\frac{\mathcal{W}_c(\tau)-\mathcal{W}_\infty(\tau)}{\langle p\rangle^3}\bigg\|_{C_{x,p}}&\lesssim \frac{\eta_0}{c^2}+\big\|(1+\tau)(E_c-E_\infty)(\tau)\big\|_{L_{\tau}^1([0,t]; L_x^\infty)}\\
&\lesssim \frac{\eta_0}{c^2}+\eta_0\sup_{0\leq \tau\leq t}\bigg\|\frac{\mathcal{W}_c(\tau)-\mathcal{W}_\infty(\tau)}{\langle p\rangle^3}\bigg\|_{C_{x,p}},
\end{aligned}$$
which yields \eqref{eq: nr limit for wave operators} as $0<\eta \ll 1$ is  sufficiently small. Then, applying \eqref{eq: nr limit for wave operators} to the inequalities in Lemma \ref{lemma: difference between force fields} and its proof provides the other two inequalities, namely \eqref{eq: nr limit for Vlasov} and \eqref{eq: nr limit for force field}.
\end{proof}

\subsection{Non-relativistic limit for scattering states}

Finally, we prove that scattering states of the relativistic system converge as $c \to \infty$ to their corresponding non-relativistic limiting states.

\begin{proof}[Proof of Theorem \ref{T3}]
First, we use \eqref{eq: nr limit for wave operators} and \eqref{eq: rate of convergence to the wave operator} to combine the non-relativistic limit of the finite-time wave operator and the scattering estimate for the wave operator to find 
$$\begin{aligned}
&\big\|\langle p \rangle^{-3}\big(\mathcal{W}_c^+(x,p)-\mathcal{W}_\infty^+(x,p)\big)\big\|_{C_{x,p}}\\
&\leq \big\|\langle p \rangle^{-3}\big(\mathcal{W}_c(t)(x,p)-\mathcal{W}_\infty(t)(x,p)\big)\big\|_{C_{x,p}}+\big\|\mathcal{W}_c(t)(x,p)-\mathcal{W}_c^+(x,p)\big\|_{C_{x,p}}\\
&\quad+\big\|\mathcal{W}_\infty(t)(x,p)-\mathcal{W}_\infty^+(x,p)\big\|_{C_{x,p}}\\
&\lesssim \frac{\eta_0}{c^2}+\frac{1}{(1+t)^\alpha}.
\end{aligned}$$
Note that in the above bound, the implicit constant does not depend on $c \in [1,\infty]$ or $t\geq0$. Therefore, taking $t\to\infty$, we prove the non-relativistic limit for the limiting wave operators 
\begin{equation}\label{eq: non-relativistic limit of the infinite-time wave operator}
\Big\|\frac{1}{\langle p \rangle^3}\Big(\mathcal{W}_c^+(x,p)-\mathcal{W}_\infty^+(x,p)\Big)\Big\|_{C_{x,p}}\lesssim \frac{\eta_0}{c^2}.
\end{equation}

Next, we show the convergence for the scattering state $f_c^+$. Indeed, by construction, we have
$$f_c^+=f^0\circ (\mathcal{W}_c^+)^{-1} \qquad \mathrm{and} \qquad f_\infty^+=f^0\circ (\mathcal{W}_\infty^+)^{-1} $$
so that
$$f_c^+=f^0\circ (\mathcal{W}_c^+)^{-1}=f^0\circ(\mathcal{W}_\infty^+)^{-1}\circ\mathcal{W}_\infty^+\circ(\mathcal{W}_c^+)^{-1}=f_\infty^+\circ\mathcal{W}_\infty^+\circ(\mathcal{W}_c^+)^{-1}.$$
Hence, by the volume preserving property of $\mathcal{W}_c^+$, and changing variables $(x,p)=\mathcal{W}_c^+(\tilde{x},\tilde{p})$, we have
$$\|f_c^+-f_\infty^+\|_{L_{x,p}^1}=\|f_\infty^+\circ\mathcal{W}_\infty^+\circ(\mathcal{W}_c^+)^{-1}-f_\infty^+\|_{L_{x,p}^1}=\|f_\infty^+\circ \mathcal{W}_\infty^+-f_\infty^+\circ\mathcal{W}_c^+\|_{L_{x,p}^1}.$$
Then, as in \eqref{Wtheta} and \eqref{eq: interpolated wave operator}, introducing the interpolated operator
$$\overline{\mathcal{W}}^\theta(x,p)=\theta\mathcal{W}_c^+(x,p)+(1-\theta)\mathcal{W}_\infty^+(x,p)$$
further yields
$$ \left |f^+_\infty \Big( \mathcal{W}_\infty^+(x,p) \Big)  - f^+_\infty \Big( \mathcal{W}_c^+(x,p)\Big) \right | 
= \left |\int_0^1 \frac{d}{d\theta} \left [f_\infty^+ \left (\overline{\mathcal{W}}^\theta(x,p) \right ) \right ] d \theta \right |.$$
Applying \eqref{eq: non-relativistic limit of the infinite-time wave operator}, we find
$$\begin{aligned}
\|f_c^+(x,p)-f_\infty^+(x,p)\|_{L_{x,p}^1}&\leq\int_0^1\Big\|\nabla_{(x,p)}f_\infty^+\Big(\overline{\mathcal{W}}^\theta(x,p)\Big)\Big(\mathcal{W}_c^+(x,p)-\mathcal{W}_\infty^+(x,p)\Big)\Big\|_{L_{x,p}^1}d\theta\\
&\lesssim\frac{1}{c^2}\int_0^1\Big\|\langle p \rangle^3 \nabla_{(x,p)}f_\infty^+\Big(\overline{\mathcal{W}}^\theta(x,p)\Big)\Big\|_{L_{x,p}^1}  d\theta.
\end{aligned}$$
Next, we perform a change of variables via $(\tilde{x},\tilde{p})=\overline{\mathcal{W}}^\theta(x,p)=(\overline{\mathcal{W}}_1^\theta(x,p), \overline{\mathcal{W}}_2^\theta(x,p))\in \mathbb{R}^3\times\mathbb{R}^3$. Indeed, $\overline{\mathcal{W}}^\theta:\mathbb{R}^6\to\mathbb{R}^6$ is invertible and $|\textup{det}(\nabla_{(x,p)}\overline{\mathcal{W}}^\theta)-1|\lesssim\eta_0$, due to the continuity of the determinant operator, as Lemma \ref{lemma: explicit formula for the classical wave operator} guarantees
$$\begin{aligned}
\|\nabla_{(x,p)}\overline{\mathcal{W}}^\theta-\mathbb{I}_6\|&\leq\theta \|\nabla_{(x,p)}\mathcal{W}_c^+-\mathbb{I}_6\|+(1-\theta)\|\nabla_{(x,p)}\mathcal{W}_\infty^+-\mathbb{I}_6\|\lesssim\eta_0,\\
|\overline{\mathcal{W}}^\theta(x,p)-(x,p)|&\leq \theta |\mathcal{W}_c^+(x,p)-(x,p)|+(1-\theta)|\mathcal{W}_\infty^+(x,p)-(x,p)|\lesssim\eta_0.
\end{aligned}$$
Hence, it follows that 
$$\begin{aligned}
\|f_c^+-f_\infty^+\|_{L_{x,p}^1}&\lesssim\frac{1}{c^{2}}\int_0^1\Big\|\langle \overline{\mathcal{W}}_2^\theta(x,p) \rangle^3 \nabla_{(x,p)}f_\infty^+\Big(\overline{\mathcal{W}}^\theta(x,p)\Big)\Big\|_{L_{x,p}^1}  d\theta\\
&\lesssim\frac{1}{c^{2}}\int_0^1\big\|\langle \tilde{p} \rangle^3 \nabla_{(\tilde{x},\tilde{p})}f_\infty^+\big\|_{L_{\tilde{x},\tilde{p}}^1}  d\theta=\frac{1}{c^{2}}\big\|\langle p \rangle^3 \nabla_{(x,p)}f_\infty^+\big\|_{L_{x,p}^1}.
\end{aligned}$$
Finally, we estimate the remaining derivatives of the scattering state $\nabla_{(x,p)}f_\infty^+$ in terms of the initial data. Taking the derivative of the equality $f^+_\infty(x,p)=f^0((\mathcal{W}^+_\infty)^{-1}(x,p))$ and using Lemma \ref{lemma: explicit formula for the classical wave operator}, we obtain
$$\begin{aligned}
\big\|\langle p \rangle^3 \nabla_{(x,p)}f_\infty^+\big\|_{L_{x,p}^1}&\lesssim\Big\|\langle p \rangle^3\nabla_{(x,p)}f^0\Big((\mathcal{W}^+_\infty)^{-1}(x,p)\Big)\nabla_{(x,p)}\Big((\mathcal{W}_\infty^+)^{-1}\Big)(x,p)\Big\|_{L_{x,p}^1}\\
&\lesssim\Big\|\langle p \rangle^3\nabla_{(x,p)}f^0\Big((\mathcal{W}^+_\infty)^{-1}(x,p)\Big)\Big\|_{L_{x,p}^1}.
\end{aligned}$$
Then, changing variables by letting $(x',p')=(\mathcal{W}^+_\infty)^{-1}(x,p)$ and noting that $|\mathcal{P}_\infty^+(x',p')-p'|\lesssim\eta_0$ due to Lemma \ref{lemma: explicit formula for the classical wave operator}, we find
$$\big\|\langle p \rangle^3 \nabla_{(x,p)}f_\infty^+\big\|_{L_{x,p}^1}\lesssim\|\langle \mathcal{P}_\infty^+(x',p')\rangle^3\nabla_{(x,p)}f^0(x',p')\|_{L_{x',p'}^1}\lesssim\|\langle p'\rangle^3\nabla_{(x,p)}f^0(x',p')\|_{L_{x',p'}^1}.$$
Therefore, including this within the above computation, we conclude 
$$\|f_c^+-f_\infty^+\|_{L_{x,p}^1}\lesssim\frac{1}{c^{2}}\|\langle p \rangle^3\nabla_{(x,p)}f^0\|_{L_{x,p}^1},$$ 
and the proof is complete.
\end{proof}

\end{document}